\documentclass[11pt]{amsart}

\usepackage{amsmath}
\usepackage{amssymb}
\usepackage{amsfonts}
\usepackage{amsthm}
\usepackage{enumerate}

\usepackage{graphicx}
\usepackage{verbatim}
\newcommand{\ls}{\lesssim}

\newcommand{\wh}{\widehat}

\newcommand{\E}{\mathcal E}

\newcommand{\F}{\mathcal F}

\newcommand{\ci}[1]{_{{}_{\!\scriptstyle{#1}}}}
\newcommand{\cin}[1]{_{{}_{\scriptstyle{#1}}}}

\newcommand{\Be}{\begin{equation}}
\newcommand{\Ee}{\end{equation}}

\newcommand{\Bm}{\begin{multline}}
\newcommand{\Em}{\end{multline}}
\newcommand{\Bea}{\begin{eqnarray}}
\newcommand{\Eea}{\end{eqnarray}}
\newcommand{\Beas}{\begin{eqnarray*}}
\newcommand{\Eeas}{\end{eqnarray*}}
\newcommand{\Benu}{\begin{enumerate}}
\newcommand{\Eenu}{\end{enumerate}}
\newcommand{\Bi}{\begin{itemize}}
\newcommand{\Ei}{\end{itemize}}

\def\sh{{\text{sh}}}
\def\lo{{\text{lg}}}
\def\o{\circ}

\def\intslash{\rlap{\kern  .32em $\mspace {.5mu}\backslash$ }\int}
\def\qsl{{\rlap{\kern  .32em $\mspace {.5mu}\backslash$ }\int_{Q_x}}}

\def\vth{\vartheta}

\def\Q{\mathcal Q}

\def\emph#1{{\it #1 }}
\def\diam{{\text{\it  diam}}}

\def\ga{\gamma}

\def\eg{{\it e.g. }}
\def\cf{{\it cf}}

\def\supp{{\text{\rm supp}}}
\def\rad{{\text{\it rad}}}
\def\sph{{\text{sph}}}

\def\inn#1#2{\langle#1,#2\rangle}
\def\biginn#1#2{\big\langle#1,#2\big\rangle}

\def\noi{\noindent}

\def\meas{{\text{\rm meas}}}

\def\lc{\lesssim}

\def\eps{\varepsilon}

\def\ka{\kappa}
            
\def\la{\lambda}

\def\vphi{\varphi}

\def\fa{{\mathfrak {a}}}

\def\bbC{{\mathbb {C}}}

\def\bbN{{\mathbb {N}}}

\def\bbR{{\mathbb {R}}}

\def\bbZ{{\mathbb {Z}}}

\def\cA{{\mathcal {A}}}

\def\cE{{\mathcal {E}}}
\def\cF{{\mathcal {F}}}

\def\cM{{\mathcal {M}}}

\def\cQ{{\mathcal {Q}}}
\def\cR{{\mathcal {R}}}
\def\cS{{\mathcal {S}}}

\def\cU{{\mathcal {U}}}

\def\cW{{\mathcal {W}}}
\def\cX{{\mathcal {X}}}
\def\cY{{\mathcal {Y}}}

\def\Q{{\hbox{\bf Q}}}

 at 10 true pt

\def\be#1{\begin{equation}\label{ #1}}
\def\endeq{\end{equation}}
\def\endal{\end{align}}
\def\bas{\begin{align*}}
\def\eas{\end{align*}}
\def\bi{\begin{itemize}}
\def\ei{\end{itemize}}

\def\eps{\varepsilon}

\def\emph#1{{\it #1}}
\def\textbf#1{{\bf #1}}

\theoremstyle{plain}
   \newtheorem{theorem}{Theorem}[section]

   \newtheorem{proposition}[theorem]{Proposition}
   
   \newtheorem{lemma}[theorem]{Lemma}
   \newtheorem{corollary}[theorem]{Corollary}

   \newtheorem{theorem*}{Theorem}

\theoremstyle{remark}
   \newtheorem{remark}[theorem]{Remark}

\theoremstyle{definition}
   \newtheorem{definition}[theorem]{Definition}

\numberwithin{equation}{section}

\begin{document}

\title{Radial Fourier multipliers in high dimensions}

\author[Y. Heo \ \ \ F. Nazarov \ \ \ A. Seeger]{Yaryong Heo \ \ \ F\"edor Nazarov  \ \ \  Andreas Seeger}

\address{
Department of Mathematics\\ University of Wisconsin-Madison\\Madison, WI 53706, USA}
\email{heo@math.wisc.edu}
\email{nazarov@math.wisc.edu}
\email{seeger@math.wisc.edu}


\begin{thanks} {Y.H. supported by  Korea Research Foundation Grant KRF-2008-357-C00002.
F.N. supported in part by NSF grant 0800243.
A.S. supported in part by NSF grant 0652890.}
\end{thanks}

\dedicatory{In memory of Brent Smith}

\begin{abstract} Given a fixed $p\neq 2$,  we prove a simple and effective
 characterization
of all radial multipliers of $\cF L^p(\Bbb R^d)$, provided that the dimension
$d$ is
sufficiently large.
The method also yields  new  $L^q$ space-time
regularity results for solutions of the wave equation in high dimensions.
\end{abstract}


\maketitle

\section*{Introduction}

In this paper we study convolution operators with  radial kernels  acting
on functions defined in $\bbR^d$.  These can also be described as
Fourier multiplier
transformations  $T_m$ defined by
$$
\widehat {T_m f} =m\widehat f,
$$
with radial $m$.
The main question we will be interested in is when the operator $T_m$
is
bounded on $L^p(\bbR^d)$, $1\le p<\infty$.
By duality, the boundedness of $T_m$ on $L^p$ is equivalent
to its boundedness on $L^{p'}$ where $\frac 1p+\frac 1{p'}=1$, so we
may restrict
ourselves to the range $1\le p\le 2$.

A simple characterization
of convolution operators  bounded on $L^p$ (whe\-ther radial or not)
is known only in two cases: $p=1$ and $p=2$; namely,
boundedness on $L^1$ holds if and only if the convolution kernel is a
finite Borel
measure and boundedness on $L^2$ holds if and only if  the multiplier is an
essentially
 bounded function (see \cite{hoer}).
It is currently widely believed that for $1<p<2$, a
full characterization of all $\cF L^p$ multipliers
in reasonable terms is impossible.  For the class of {\em radial} multipliers
we deal with in this paper, numerous sufficient conditions
for boundedness on $L^p$ have been obtained in the literature.
Many of them are in some or another sense close to being necessary
(\cf. \cite{cs},  \cite{bourg}, \cite{lee},  \cite{cgt}, \cite{wolff},
\cite{ms-adv},
and references in those papers) but  no nice necessary and sufficient
conditions have been known. However, recently,
Garrig\'os  and the third author  \cite{gs}
obtained a perhaps  surprising
characterization of the radial multiplier transformations that are bounded
on  the invariant subspace $L^p_\rad$ of radial $L^p$ functions in the range
$1<p<\frac{2d}{d+1}$ (which is optimal for their result). This raised the question
whether the necessary and sufficient conditions
in \cite{gs} actually give a characterization of the radial
multiplier
transformations bounded on  the entire space $L^p(\bbR^d)$.
The main result of the  present paper is to  show that this  is
indeed the case if the dimension is sufficiently large, namely
if $d>\frac{2+p}{2-p}$, $1<p<2$.




\section{Statement of results}
\label{statement}

\begin{theorem}
\label{mainthm}
Let $d\ge 4$, $1<p<p_d:=\frac{2d-2}{d+1}$,
 and let  $m$ be  radial.
Fix an arbitrary Schwartz function
$\eta$ that is not identically $0$.
Then
\begin{equation}\label{equiv}
\big\|T_m\big\|\cin{L^p\to L^p}\,\asymp \,\sup_{t>0}\,t^{d/p}\big\|T_m[\eta(t\cdot)]\big\|\ci{L^p}\,.
\end{equation}
\end{theorem}

The finiteness of the right hand side is, obviously, necessary for
the $L^p$ boundedness, and the main result here is that it is also
sufficient. The constants implicit in this characterization depend
(of course) on the choice of $\eta$.  The condition in \eqref{equiv} is equivalent to  $\sup_{t>0} \|\cF^{-1}[ m(t\cdot) \widehat \eta ]\|_p<\infty$.
If one chooses $\eta$ to be
radial and such that $\widehat \eta$ is compactly supported away
from the origin,
then one recovers one of the characterizations for
$L^p_\rad$ boundedness in \cite{gs}. Consequently, in the given
range $L^p$ boundedness is equivalent to $L^p_\rad$ boundedness.
We refer the reader to
\cite{gs} for other equivalent formulations.

One special situation is worth mentioning here. Namely,
if  $m$ is compactly supported away from
the origin  and $1<p<p_d$, then
 the convolution operator is bounded on $L^p(\bbR^d)$ if and only if the
(radial) convolution kernel $\widehat m$ belongs to $L^p(\bbR^d)$.

We have no reason to believe that the range for $p$ in Theorem \ref{mainthm} is
even close to the optimal one. It is conceivable that the characterization
holds in low dimensions or even  in the optimal range $p<\frac{2d}{d+1}$,
but proving that will certainly
require new ideas.
We also emphasize that the theorem gives no improvements
for the Bochner-Riesz multiplier problem that
is by now understood in the range $p<\frac{2d+4}{d+4}$, $d\ge 2$ (see \cite{cs},
\cite{lee}). Our result just goes in a
different direction:
it applies to all, however irregular,
radial kernels  and it is to be expected that, using some
additional structural or regularity conditions, one may get some
better range of $p$ for each particular case. Nevertheless, our technique
does yield some improvements upon the existing results in the so-called
local smoothing problem for the wave equation in high dimensions.
This concerns inequalities of the form
\begin{equation}\label{lsmineq}
\Big(\int_{I}\|e^{it\sqrt{-\Delta}}f\|_q^q \,dt\Big)^{1/q}
\le C\ci I \|f\|\cin{L^{q}_{\alpha}},
\end{equation}
for $q>2$; here $I$ is a  compact interval and
$L^{q}_\alpha(\bbR^d)$ denotes
the usual Sobolev (or potential) space where $q$ is the Lebesgue exponent and
$\alpha$ is
the number of derivatives.
Sharp $L^q$-Sobolev inequalities for fixed time were obtained by
  Miyachi \cite{miyachi} and Peral \cite{peral}; they showed
that the operator
$e^{it\sqrt{-\Delta}}$ maps $L^q_\beta (\bbR^d)$ into $L^q(\bbR^d)$
provided that
$\beta\ge (d-1)|1/2-1/q|$, $1<q<\infty$.
In \cite{sogge} Sogge
raised the question
whether  the averaged  inequality \eqref{lsmineq}
could hold with a
gain of almost $1/q$ derivatives compared to the fixed time estimate,  {\it i.e.},
with $\alpha> \alpha(q)= d(1/2-1/q)-1/2$, in the best possible range
$q>2d/(d-1)$ for such an estimate.
 This conjecture is at the top of a tree of other conjectures
in harmonic analysis (including the  cone multiplier,
Bochner-Riesz, Fourier-restriction and Kakeya conjectures)
 and the relation between  the different
questions is discussed, for example,  in \cite{taorestrbr}. The
current techniques  seem to be  insufficient to settle this
problem, as well as many of its consequences,  in the full range
of $q$'s. Some evidence for the smoothing conjecture can be found in
\cite{ms-adv} where the analogous question  for the
$L^q_\rad(L^2_\sph)$ scale  of spaces is settled.
For the $L^q$ spaces even partial results proved to be
rather hard and the first  result was  obtained
by Wolff \cite{wolff}; he established, in a deep and fundamental
paper, the validity of Sogge's conjecture
in two dimensions for the range $q>74$. Versions  of this result for
the higher dimensional cases were obtained by \L aba and Wolff
 \cite{LW}
and further  improvements on the range of $q$'s
 are in \cite{gs-lsm}, \cite{gss};  it is now known
that Wolff's main $\ell^q(L^q) \to L^q$
inequality for plate decompositions of cone multipliers, which implies
 \eqref{lsmineq} for $\alpha>\alpha(q)$, holds with  $q>20$ if $d=2$ and  $q> 2+ \frac{8}{d-2}\frac{2d+1}{2d+2}$ if
$d\ge 3$ ({\it cf.} \cite{gss}).

We improve the current results on the smoothing problem
 in two ways.
First, we widen  the range
in dimensions $d\ge 5$.
Secondly we
strengthen Sogge's conjecture to obtain an  endpoint
result  in \eqref{lsmsobolev} in dimensions $d\ge 4$.

\begin{theorem}\label{lsmsobolev} Suppose $d\ge 4$ and
$q> q_d:=2+\frac{4}{d-3}$.
Then   there is a constant $C_{q,d}$ such  that for all $L>0$,
\begin{equation}\label{unscaledSob}
\frac{1}{2L}\int_{-L}^L \big\|e^{it\sqrt {-\Delta}} f \big\|_q^q \,dt
 \le C_{q,d}^q \big\|(I-L^2\Delta)^{\alpha/2} f\big \|_q^q
\end{equation}
holds for  $\alpha= \alpha(q)= d(1/2-1/q)-1/2$.
\end{theorem}
We remark that this result can be strengthened  further by using
suitable Triebel-Lizorkin spaces, see  \S\ref{lsmsect}. A similar phenomenon
occurs for solutions of Schr\"odinger type equations, see \cite{rs}.

A downside of our method
is, of course,
that it  currently does not yield $L^p$ results in two and three  dimensions.
However, when it does apply,  it
is somewhat simpler than the induction on scales methods
introduced by Wolff.
We also remark that we do not improve on the current range
of the above\-mentioned
 Wolff inequality for plate decompositions, which
has other applications and is interesting in its own right.

\medskip


\noi{\it Structure of the paper.}
In \S\ref {modelsect} we explain the basic idea of the paper,
which is that  weak orthogonality properties may be combined with
support size estimates
 to prove satisfactory $L^p$ bounds. Here we also state a
basic interpolation lemma which is related to the Marcinkiewicz theorem and will be used throughout the paper.
The main section is \S\ref{mainsect} where we outline  the proof of
a discretized version of Theorem \ref{mainthm} for a fixed scale.
A  crucial  $L^2$ estimate needed for this proof is done in \S\ref{L2bd}. The characterization of $L^p$ boundedness for radial multipliers that are compactly supported away from the origin is proved in
\S\ref{compmult}. In \S\ref{largeradsect} we give an important refinement
 of the earlier  estimates,
which is crucial for
putting scales together. This is completed   in \S\ref{atomicsect}
where the relevant
 atomic decomposition techniques are introduced and
applied. The proof of Theorem \ref{mainthm} is concluded in \S\ref{concl}.
In \S\ref{extvar} we state an extension to $H^p$ spaces, $p\le 1$,
 which holds for dimensions $d\ge 2$; moreover we obtain
Lorentz space bounds (including weak type $(p,p)$ inequalities).
The
last section \S\ref{lsmsect} contains the proof of
 (a somewhat strengthened version of)   Theorem \ref{lsmsobolev}.

\medskip

\noi{\it Notation.}
For two quantities $A$ and $B$, we shall write
$A\lc B$ if $A\le CB$ for some  positive
constant $C$,
depending on the dimension and possibly other parameters apparent from the context, for instance Lebesgue exponents.
We write $A \asymp B$ if $A\lc B$ and $B\lc A$.
The cardinality of a finite set $\cE$ is denoted by $\#\cE$.
The $d$-dimensional  Lebesgue measure  of a set $E\subset \bbR^d$ will be denoted by
$\meas(E)$ or  by $|E|$.

\medskip

\noindent{\it Remark.}
This paper is a descendant of the unpublished manuscript \cite{na-se}
with the same title in which  Theorems \ref{mainthm} and
\ref{lsmsobolev} were proved in dimensions $d\ge 5$ for  slightly
smaller ranges of $p$ and $q$.
The approach in the present paper simplifies the one in
\cite{na-se} and was inspired in part by an idea in \cite{heo}.
The authors would like to thank
 Gustavo Garrig\'os and Keith Rogers for  their  comments on various
preliminary versions of \cite{na-se}.

\section{$L^2$ bounds versus support: A simple model case}\label{modelsect}
Since we do not know how to exploit  cancellations in $L^p$ directly,
 we use the  strategy of controlling the $L^2$ norm and the size of the support
simultaneously to get our $L^p$
bounds. We start with describing
 a simple model case for which we have some limited
orthogonality,  but not enough to prove a favorable $L^2$ bound.

\begin{lemma}\label{modellemma}
Suppose we are given  a finite number of  complex-valued $L^2$-functions
$\{f_z\}$ indexed by $z\in \bbZ^d$
such  that each function $f_z$
is supported in a cube $Q_z$ of sidelength $1$.
Suppose also that the family  $\{f_z\}$ satisfies
\begin{equation} \label{betaorth}
  |\inn{f_z}{f_{z'}} | \le (1+|z-z'|)^{-\beta},
\end{equation}
for some $\beta \in (0,d)$.
Then for $p<  \frac{2d}{2d-\beta}$,
\begin{equation} \label{lpimpr}
\Big\|\sum_{z} a_z f_z\Big\|_p \lc\,\Big(\sum_{z}|a_z|^p
\Big)^{1/p}\,.
\end{equation}
\end{lemma}

The implicit constant in \eqref{lpimpr} depends on $d$, $\beta $ and
$p$.
Note that \eqref{lpimpr} is trivial for $p\le 1$.
We remark that if \eqref{betaorth} were assumed  for some  $\beta>d$, then inequality
\eqref{lpimpr} would also be true for $p=2$ and thereby
for
$1<p<2$ by interpolation.
The assumption \eqref{betaorth} for $\beta<d$ is too weak to
yield the $\ell^2\to L^2$ bound.
 Instead we have to use some improved support properties
when several of the cubes $Q_z$
overlap.

\begin{proof}[Proof of Lemma \ref{modellemma}]
We shall first  prove a weaker
(so-called restricted strong type) inequality that includes the endpoint;
namely for $1\le  p\le  \frac{2d}{2d-\beta}$,
\begin{equation} \label{lessE}
\Big\|\sum_{z\in E} a_z f_z\Big\|_p \lc (\# E)^{1/p} \sup_z |a_z|,
\end{equation}
We may assume that $\sup_z|a_z|= 1$.
Let $x_z\in\bbR^d$ be the center of the  cube $Q_z$  of sidelength $1$
supporting
$f_z$. Split $\Bbb R^d$ into  nonoverlapping cubes  $J$ of sidelength $1$,  put
$E\ci J=\{z\in E:x_z\in J\}$, and define $u\ci J=\#E\ci J$
so that  $\#E =\sum_J u\ci J$.
We have to bound the $L^p$ norm of  $\sum_J F\ci J$, where
$F\ci J=\sum_{z\in E\ci J}a_z f_z.$

Now observe that at each point $x\in\bbR$, at most $3^d$ of the
 functions $F\ci J$
 can be non-zero
simultaneously. Therefore
$$
\Bigl\|\sum_J F\ci J\Bigr\|\ci{p}^p\le 3^{dp}\sum_J \|F\ci J\|\ci{p}^p\,.
$$
Now, according to our weak orthogonality assumption about the functions $f_z$, we have
\begin{align*}
\big\|F\ci J\big\|\ci{2}^2&\le \sum_{z\in E\ci J}\sum_{z'\in E\ci J}(1+|z-z'|)^{-\beta}
\\&
\le \sum_{z\in E\ci J}\sum_{\substack{z':|z-z'| \le \sqrt d u_J^{1/d}}}
(1+|z-z'|)^{-\beta}
\lc u_J^{2-\frac{\beta}{d}}.
\end{align*}
The measure of the support of $F\ci J$ is at most $2^d$ and therefore,
by H\"older's inequality,
$\|F\ci J\|_p\lc \|F\ci J\|_2$.
Hence
$$\Big\|\sum_J F\ci J\Big\|_p \lc \Big(\sum_J \big\|F\ci J\big\|_2^p\Big)^{1/p}
\lc \Big(\sum_J
u_J^{(2-\frac \beta d)\frac p2}
\Big)^{1/p}
$$
and if
$(2-\frac \beta d)\frac p2\le 1$,
then the last expression is bounded by
$(\sum_J u\ci J)^{1/p}\le (\#E)^{1/p}$.
This yields  \eqref{lessE}.

The improved bound \eqref{lpimpr} can be deduced by  using interpolation theorems for
Lorentz spaces
 (see \cite{stw}, ch. V).
Consider the operator on sequences $\fa=\{a_z\}_{z\in \bbZ^d}$, given by
$T[\fa]=\sum_{z} a_z f_z$. Then \eqref{lessE} states that
$T$ maps the Lorentz space $\ell^{p,1}$ to $L^p$, for
$p\le  2d/(2d-\beta)$
and, by interpolation, one deduces the
inequality \eqref{lpimpr} in the open range
$p<  2d/(2d-\beta)$
\end{proof}

We wish to give  a direct proof of the last  interpolation result
 based on a dyadic
interpolation lemma, which will be  frequently  used in this
 paper. For closely
related considerations see also the expository note \cite{tao-lor} by Tao.

\begin{lemma} \label{dyadicinterpol}
Let $0<p_0<p_1<\infty$.
Let $\{F_j\}_{j\in\bbZ}$ be a sequence of measurable functions on a measure space
$\{\Omega, \mu\}$,
and let $\{s_j\}$ be a sequence of nonnegative numbers.
Assume that, for all $j$,
the inequality
\begin{equation}\label{pnuassumption}
\|F_j\|_{p_\nu}^{p_\nu}\le 2^{j p_\nu} M^{p_\nu} s_j
\end{equation}
holds  for $\nu=0$ and $\nu=1$.
Then for all $p\in (p_0,p_1)$, there is a constant $C=C(p_0,p_1,p)$ such
 that
\begin{equation}\label{dyadicinterpolconclusion}
\Big\| \sum_j F_j\Big\|_p^p \le C^p M^p \,\sum_j 2^{jp} s_j
\end{equation}
There is an analogous statement for the case $p_0=0$  where the assumption
\eqref{pnuassumption} for $\nu=0$ is replaced with
$\meas(\{x:F_j(x)\neq 0\})\le s_j$, and the conclusion
\eqref{dyadicinterpolconclusion} holds for $0<p<p_1$.
\end{lemma}

To see how this is used to  derive
\eqref{lpimpr} from
\eqref{lessE},
we consider the sets of indices
$E_j=\{z\in \bbZ^d:2^{j-1}<|a_z|\le 2^j\}$ and define
$F_j=\sum_{z\in E_j}a_z f_z$. Then
$\|F_j\|\cin{L^p}^p\ls 2^{jp} \# E_j$
for all
$p\in (0,2d/(2d-\beta)]$
by \eqref{lessE}.
Thus
Lemma \ref{dyadicinterpol} immediately yields
$\|\sum_z a_z f_z\|_{p}^p=
\|\sum_j F_j\|_{p}^p\ls \sum_j 2^{pj}\#E_j
\ls \sum_z |a_z|^p
$
for all $p<2d/(2d-\beta).$

\begin{proof}[Proof of Lemma \ref{dyadicinterpol}]
First, replacing $F_j$ by $M^{-1}F_j$, we can reduce the statement
to the case $M=1$. Now,  for $n\in \bbZ$, denote by $E_{j,n}$
the set where $2^{j+n}\le|F_j|<2^{j+n+1}$ and put
$F_{j,n}=\chi\ci{E_{j,n}} F_j$.
Then  $F_j=\sum_{n\in \bbZ}
F_{j,n}$. Observe that if $b_j$ is any numerical sequence such
that for every $j$, the absolute value of $b_j$ either is $0$ or
belongs to $[2^j,2^{j+1})$, then $|\sum_j b_j|^p\lc \sum_j
|b_j|^p$.
 Applying this observation to  $2^{-n}\sum_j F_{j,n}$, we see that for fixed $n$ and $x$,
$$\Big|\sum_j F_{j,n}(x)\Big| \lc \Big(\sum_j|F_{j,n}(x)|^p \Big)^{1/p}
$$
and therefore
$$
\Bigl\|
\sum_j F_{j,n}
\Bigr\|_p^p\lc \sum_j \|F_{j,n}\|_p^p\lc \sum_j 2^{(j+n)p}
\meas\big(\{x: |F_j|\ge 2^{j+n}\}\big).
$$
By Chebyshev's inequality,
$$
\meas\big(\{x: |F_j|\ge 2^{j+n}\}\big)
\le \min\{2^{-p_0n},2^{-p_1n}\}s_j\,.
$$
Thus,
$$
\Bigl\|
\sum_j F_{j,n}
\Bigr\|_p\ls 2^{-\sigma|n|/p}\Big(\sum_j 2^{jp}s_j\Big)^{1/p}
$$
where $\sigma=\min\{p_1-p,p-p_0\}$.
We sum in $n$ to get the statement of the lemma for the case
$p_0>0$. The case $p_0=0$ is very similar and is left to the reader.
\end{proof}



\section{The main inequality}\label{mainsect}
In this section we shall prove the main inequality of this paper, which turns
out to be the key estimate for the case when our  multiplier has compact
support away from the origin; this application is discussed at the
end of the section.

In what follows, we denote by $\sigma_r$ the surface measure on the
$(d-1)$-dimensional  sphere of radius $r$ centered at the origin.1
We shall denote by $\psi_\o$ a fixed radial $C^\infty$ function that
is  compactly supported in a ball of radius $1/10$ centered at the
 origin,
and
whose Fourier transform $\wh\psi_\o$ vanishes to high order (say,  $20 d$)
at the origin.
We set $\psi=\psi_\o*\psi_\o$.

Consider a $1$-separated set $\cY$ of points in $\bbR^d$  and
a $1$-separated set $\cR$  of radii $\ge 1$. Also set
$$\cR_k=\cR\cap[2^k, 2^{k+1}),\qquad k\ge 0.$$ For $y\in \cY$ and
  $r\in \cR$,
  define
\begin{equation} \label{gyr}
F_{y,r}  = \sigma_r * \psi(\cdot-y).
\end{equation}

\begin{proposition}\label{main} Let $\cE$ be a finite subset of $\cY\times \cR$ and let
$\cE_k=\cE\cap (\cY\times\cR_k$).
Let $c: \cE\to \bbC$ be a function satisfying
$|c(y,r)|\le 1$ for all $(y,r)\in \cE$.
Then, for $p<p_d=\frac{2d-2}{d+1}$,
\begin{equation}
\Big\| \sum_{(y,r)\in \cE} c(y,r) F_{y,r}\Big\|_p^p \lc
\sum_k 2^{k(d-1)} \,\# \cE_k;
\end{equation}
here the implicit constant depends only on $p$, $d$ and $\psi$.
\end{proposition}


Proposition \ref{main} implies   stronger estimates, namely,

\begin{corollary}\label{mainlp}
For $F_{y,r}$ as in \eqref{gyr}
and $p<p_d=\frac{2d-2}{d+1}$,
\begin{equation} \label{mainlpassertion}
\Big\| \sum_{(y,r)\in\cY\times\cR}
 \ga(y,r) F_{y,r}\Big\|_p \lc
\Big(\sum_{(y,r)\in\cY\times\cR}|\ga(y,r)|^p r^{d-1}\Big)^{1/p}.
\end{equation}
Also,
\begin{equation}\label{mainlpcont}
\Big\| \int_{\bbR^d}\int_1^\infty
h(y,r) F_{y,r} \, dr dy \Big\|_p \lc
\Big(\int_{\bbR^d}\int_1^\infty |h(y,r)|^p r^{d-1}dr\, dy \Big)^{1/p}.
\end{equation}
\end{corollary}

\begin{proof}
Denote by $\cE^j$, $j\in\bbZ$,
the set of all $(y,r)\in \cY\times \cR$ for which $2^{j-1}< |\ga(y,r)|\le 2^j$.
By Proposition \ref{main}
 we see that  $\| \sum_{(y,r)\in \cE^j} \ga(y,r) F_{y,r}\|_p^p$
is dominated by
$C_p^p
2^{jp} \sum_{(y,r)\in \cE^j}  r^{d-1}$
for all $p<p_d$, and \eqref{mainlpassertion} follows
 by the
dyadic interpolation Lemma \ref{dyadicinterpol}.

To prove \eqref{mainlpcont}, we write $y=z+w$, where $z\in \bbZ^d$,
$w\in Q_\circ:=[0,1)^d$ and $r=n+\tau$ where $n\in \bbN$,
  $0\le\tau<1$.
Then, by Minkowski's inequality, the left hand side of
\eqref{mainlpcont} is dominated by
\begin{multline*}
\iint_{Q_\circ\times[0,1)}
\Big\|\sum_{z\in \bbZ^d}\sum_{n=1}^\infty h(z+w,n+\tau) F_{z+w,n+\tau}
\Big\|_p dw\, d\tau\\
\lc
\iint_{Q_\circ\times[0,1)} \Big(\sum_{z\in \bbZ^d}\sum_{n=1}^\infty
|h(z+w,n+\tau)|^p (n+\tau)^{d-1}\Big)^{1/p} dw\, d\tau.
\end{multline*}
Now \eqref{mainlpcont} follows by H\"older's inequality.
\end{proof}

If
$h$ has a tensor product structure,
namely, $h(y,r)=g(y)\beta(r)$,
then the expression $\iint h(y,r) F_{y,r}dr$ can be interpreted as
a convolution of a radial kernel with $g$.
In \S\ref{compmult}  we shall see  how
 this model case implies
the version of  our theorem
for radial multipliers that are
compactly supported away from the origin.

We shall present the proof of Proposition \ref{main}
(leaving one part to the next section).

\medskip
\noi{\bf Estimates for scalar products.}
We aim at a good $L^2$ estimate for $\sum c_{y,r} F_{y,r}$
and  make use of some
(albeit weak) orthogonality property of the summands.
This property   is  expressed by
\begin{lemma}\label{scalarlemma}
For any choice of $r, r'>1$ and $y, y'\in \bbR^d$
\begin{equation}\label{scalarprod}
\big|\biginn{F_{y,r}}{F_{y'\!,r'}}\big|
\lc \frac{(rr')^{\frac{d-1}{2}}}{(1+|y-y'|+|r-r'|)^{\frac{d-1}{2}}}.
\end{equation}
\end{lemma}
\begin{proof} Note that $\sigma_r \!=\! r^{-1}\sigma_1(r^{-1}\cdot )$
in the sense of measures and that $\widehat \sigma_r(\xi)\! =\!r^{d-1}
\widehat \sigma_1(r\xi)$. Next, $\wh \sigma_1(\xi)\! =\! B_d(|\xi|)$ where
$B_d(s)= c_d s^{-(d-2)/2} J_{(d-2)/2}(s)$
(and $J_.$ denotes the  usual   Bessel
functions).  Thus $|B_d(s)|\lc (1+|s|)^{-(d-1)/2}$ (see \cite{stw}, ch. IV).
Now $\widehat \psi$ is radial and we can write  $\widehat \psi(\xi)= a(|\xi|)$ where $a$ is rapidly decaying and  vanishes to high order at the origin.
 By Plancherel's theorem, the scalar product
$\inn{F_{y,r}}{F_{y'\!,r'}}$
is equal to a constant times
\begin{align*}
&\int \widehat \sigma_r(\xi)   \widehat \sigma_{r'} (\xi) |\psi(\xi)|^2
e^{i\inn{y'-y}{\xi}} d\xi\\&=
c\, (rr')^{d-1}
\int B_d(r\rho) B_d(r'\rho) B_d(|y-y'| \rho) |a(\rho)|^2 \rho^{d-1} d\rho
\end{align*}
The decay properties of $B_d$ and the behavior of $a$ imply that
$$\big|\biginn{F_{y,r}}{F_{y'\!,r'}}\big|
\lc \frac{(rr')^{\frac{d-1}{2}}}{(1+|y-y'|)^{\frac{d-1}{2}}}
$$
which gives the claimed bound for the range
$|r-r'|\le C(1+|y-y'|).$
But if
$|r-r'|\gg 1+|y-y'|$, then $F_{y,r}$ and $F_{y'\!,r'}$ have disjoint
supports.
Thus  in this case $\inn{F_{y,r}}{F_{y'\!,r'}}=0$. The lemma is proved.
\end{proof}

\begin{remark}
Taking into account
the oscillation
of the Bessel functions, one can obtain the improved bound
\begin{equation*}
|\inn{F_{y,r}}{F_{y'\!,r'}}|\le C_N(rr')^{\frac{d-1}{2}}
(1+|y-y'|)^{-\frac{d-1}{2}}
\sum_{\pm,\pm}
\big(1+\big|r\pm r'\pm |y-y'|\big|\big)^{-N}.
\end{equation*}
We shall not use it  in our proof.
%
\end{remark}

The exponent $(d-1)/2$ in the denominator  in \eqref{scalarprod} is too small
to use orthogonality in a straightforward way; this is analogous to the
weak ortho\-gonality assumption
in Lemma \ref{modellemma}.
However if we impose a suitable density  assumption
on the sets $\cE_k$, then  we can prove a satisfactory $L^2$ bound.
To quantify this, we give a definition.

\begin{definition}
Fix $R\ge 1$ and $u\ge 1$.
Let $\cE$ be a finite  $1$-separated subset of $\bbR^d\times [R, 2R)$. We say
that $\cE$ is of {\it density type
$(u,R)$}  if
$$\#(B\cap \cE) \le u \,\diam(B)$$
for any ball $B\subset \bbR^{d+1}$ of diameter
 $\le R$.
\end{definition}

If we  drop the restriction on the diameter  then for any ball $B$ and
any set $\cE $ of density type $(u,R)$,
\begin{equation}\label{densitytype} \#(B\cap \cE)
\le C_d \Big(1+ \frac{\diam (B)}{R}\Big)^{d}\,  u \,\diam(B).\end{equation}
This is immediate from the definition.

We shall prove in section \S\ref{L2bd} the following $L^2$ inequality
based  on Lemma \ref{scalarlemma}.

\begin{lemma} \label{L2bdprop} 
Let  $u\ge 1$, and,
 for each $k\ge 0$, let $\cE_k\subset \cY\times \cR_k$ be a  set of density
type  $(u,2^k)$.
Assume that $|c(y,r)|\le 1$ for $(y,r)\in \cY\times \cR$.  Then
\begin{equation}\label{l2est}
\Big\|\sum_k \sum_{(y,r)\in \cE_k} c(y,r) F_{y,r}\Big\|_2^2
\lc  u^{\frac{2}{d-1}}\,\log(2+u)
\sum_k 2^{k(d-1)} \,\#\cE_k.
\end{equation}
\end{lemma}

%

\medskip

\newcommand{\mur}{R_{k,u} }

\noi{\bf Density decompositions of sets.}
Assume that $\E\subset\cY\times\cR$ is a finite  $1$-separated set. Let
$\cE_k=\cE\cap(\cY\times \cR_k)$ ({\it i.e.}, only radii in $[2^k, 2^{k+1})$ are involved).
We consider $u\in \cU=\{ 2^\nu, \nu=0,1,2,\dots\}$
and decompose the
sets
$\cE_k$ into subsets of density type $(u,2^k)$.

Let
$\wh \E_k(u)$ be the set of all points $(y,r)\in \cE_k$ that are
contained in some ball $B$ of radius $\rad(B)\le 2^k$  such that
\begin{equation}
\label{lowerboundsforballs}
\#(\cE_k\cap B) \ge u\, \rad(B).
\end{equation}
Also set
$$
\E_k(u)=\wh \E_k(u)\setminus
\bigcup_{\substack{u'\in \cU\\u'>u}}\wh \E_k(u')\,.
$$
Finally set   $\cE(u)= \bigcup_k \cE_k(u).$

\begin{lemma} \label{udec} The sets $\E(u)$  have the following
  properties.

(i) $\E=\bigcup_{u\in \cU}\E(u) = \bigcup_{u\in \cU}\bigcup_{k\ge 0}
 \cE_{k}(u)$ and the unions are  disjoint.

(ii) If $B$ is any ball of radius $\le 2^k$
containing at least $u\,\rad(B)$ points of
$\cE_k$, then
$$B\cap \cE_k\subset \wh \E_k(u)\equiv
 \bigcup_{\substack {u'\in \cU \\ u'\ge u}}\E_k(u').
$$

(iii)  There are finitely many
disjoint balls $B_1,\dots, B_N$ (depending on $u$ and $k$), of
radii $\le 2^{k}$
such that
\begin{equation}\label{summingradii}
\sum_{i=1}^N \rad(B_i) \le u^{-1}\#\cE_k,
\end{equation}
and
\begin{equation}
\label{covering}
\wh \cE_k(u)\subset \bigcup_{i=1}^N B_i^*,\end{equation}
where $B_i^*$ denotes the ball with
 $\rad(B_i^*)=5\rad(B_i)$
and the same center as $B_i$.

\indent(iv)
 $\cE_k(u)$ is a set of density type $(u,2^k)$.
\end{lemma}
\begin{proof}
In order to prove (i), it suffices to observe that $\wh \cE_k(2^0)=\cE_k$ and
$\wh \E_k(u)=\varnothing$
when $u$ is sufficiently large. (ii) follows immediately from the
definition of the sets $\wh \cE_k(u)$ and $\cE_k(u)$.

To prove (iii),
cover
the set $\widehat \cE_k(u)$
 by a finite number of balls
satisfying
\eqref{lowerboundsforballs}.  We apply the Vitali  covering lemma to
 this family of balls and select {\it disjoint}  balls $B_i$,
$i=1,\dots,N(k,u,\cE)$ so that the five times dilated balls $B_i^*$
cover $\widehat \cE_k(u)$. This yields
\eqref{covering}. The inequality
\eqref{summingradii} follows from the disjointness of the selected
 balls and condition
\eqref{lowerboundsforballs}.

To prove (iv), let $(y,r)\in \cE_k(u)$. By definition $(y,r)\notin
\wh \cE_k(2u)$ and thus, for any ball  $B$ of radius $\rad(B)\le 2^{k}$,
the number of points in $\cE_k$ contained in $B$ is  less than
$2u\,\rad(B) =u\,\diam(B) $. Thus $\cE_k(u)$ is of density type $(u,2^k)$.
\end{proof}

We now set
\begin{equation}\label{Gudef}
G_{u,k}= \sum_{(y,r)\in \cE_k(u)} c(y,r) F_{y,r} \quad\text{ and }\quad
G_u= \sum_k G_{u,k}.
\end{equation}
From the support properties of $\sigma_r*\psi$ it follows immediately
that $G_{u,k}$ is supported  in a set of measure
$\lc 2^{k(d-1)} \#\cE_{k}(u)$, hence of measure $\lc 2^{k(d-1)} \#\cE_{k}$. By the properties of $\cE_{k}(u)$ we get the following improved bound.

\begin{lemma}\label{supportlemma} For all $u\in \cU$,
the Lebesgue measure of the support of $G_{u,k}$
is $\lc u^{-1} 2^{k(d-1)} \#\cE_k.$
\end{lemma}
\begin{proof}
We use \eqref{covering}.
Let $(y_i,r_i)$ be the center of $B_i^*$. Then, for every
pair $(y,r)$ contained in $B_i^*$,  the support of $c(y,r) \sigma_r*\psi(\cdot-y)$ is contained in the
annulus of width not exceeding $4 \rad(B_i^*)+1$
built on the sphere centered
at $y_i$ of radius $r_i$.
Also, note that
the estimate for the width of the annulus does not exceed the estimate for the
radius
of the sphere it is built upon, so we can  conclude that the volume of
this
annulus is $\lc 2^{k(d-1)} \rad(B_i^*)$.
Consequently the measure  of the support of $G_{u,k}$
does not exceed
$C_d 2^{k(d-1)} \sum_{i=1}^N\rad(B_i^*)$, and hence, by \eqref{summingradii}, it does not exceed $5C_d 2^{k(d-1)} u^{-1}\#\cE_k $.
\end{proof}

We now combine the $L^2$ bound of Lemma \ref{L2bdprop} and the
support bound of
Lemma \ref{supportlemma} to get an $L^p$ bound; for later reference  in
 \S\ref{largeradsect}
this is formally stated as

\begin{lemma}\label{lpbound}
Suppose $d\ge 4$.
Let $G_u$ be as in \eqref{Gudef} where the sets $\cE_{k}(u)$ are
defined using the density decomposition of $\cE_k$.
Then, for $p\le 2$,
$$\|G_u\|_p\lc  u^{-(1/p-1/p_d)}\sqrt{\log(2+u)} \Big(\sum_k 2^{k(d-1)}\#\cE_k\Big)^{1/p}.$$
\end{lemma}
\begin{proof}
By
Lemma
\ref {L2bdprop},
$\|G_u\|_2^2 \lc  \log(2+u)u^{2/(d-1)} \sum_k 2^{k(d-1)} \,\#\cE_k.$
Combining this with the support bound of  Lemma \ref{supportlemma} we obtain
\begin{align*}
\|G_u\|_p^p &\le
\big(\meas (\supp (G_{u}))\big)^{1-p/2} \|G_u\|_2^p
\\&
\le
\Big(\sum_k\meas (\supp (G_{u,k}))\Big)^{1-p/2} \|G_u\|_2^p,
\end{align*}
which is
$\lc u^{-(1-\frac p2)} \big(\log(2+u)
u^{\frac {2}{d-1}}\big)^{\frac p2}
\sum_k 2^{k(d-1)} \,\#\cE_k.$
We finally note that  $-1+\frac p2 +\frac{p}{d-1}=(\frac{1}{p_d}-\frac1p)p$, and the lemma is proved.
\end{proof}

The proof of Proposition \ref{main} is now complete since for
$p<p_d$, we can sum the bounds
for  $\|G_u\|_p$ over  $u\in \cU $.

\section{Proof of Lemma \ref{L2bdprop}}\label{L2bd}
We are working with sets $\cE_k\subset \cY\times \cR_k$, which have the
property that every ball of radius $\rho\le 2^{k}$ contains
$\lc u\rho$ points in $\cE_k$.
Let $$G_k= \sum_{(y,r)\in \cE_k} c(y,r)   F_{y,r}$$ with
$\|c\|_\infty\le 1$.
Our task is  to estimate the $L^2$ norm of $\sum_k G_k$.
We may break up this sum into ten separate sums,  each with the property
that $k$ ranges over a $10$-separated set  of
natural numbers. We shall assume this separation property in all sums involving
a $k$-summation.

It will be convenient
to avoid scalar products of expressions of  $G_k$ involving $k \lc \log(2+u)$.
Let $N(u)$ be the smallest integer larger than $10 \log_2(2+u)$.
Split the sum as
$\sum_{k\le N(u)} G_k + \sum_{k>N(u)} G_k$  and then
apply the Cauchy-Schwarz inequality.
We thus obtain
\begin{align}\notag
\Big\|\sum_k G_k\Big\|_2^2 &\lc \log(2+u)
\Big[
\sum_{k\le N(u)}\|G_k\|_2^2+\Big\|\sum_{k>N( u)} G_k\Big\|_2^2\Big]
\\
&\lc\log(2+u)
\Big[
\sum_{k}\|G_{k}\|_2^2
+2\sum_{k'> k>N(u)}\bigl|\langle
G_{k'},G_{k}
\rangle\bigr|
\Big]\,.
\label{sumofkscprod}
\end{align}
We begin  with  estimating  the double sum
$\sum_{k'> k>N( u)}\bigl|\langle
G_{k'},G_{k}\rangle\bigr|$. In this sum we have various scalar products
of $F_{y,r}$ with $F_{Y,R}$ where
$r\le R 2^{-5}$. Let us fix the pair $(Y,R)$ and examine the sum of the
absolute values of such scalar products when $(y,r)$ runs over
$\cE_k$ with $2^k<R/4$. The scalar
product $\inn{F_{y,r}}{F_{Y,R}}$ can be different from $0$ only if $y$ lies in the annulus of
width  $2^{k+1}+2$ built upon the
sphere of radius $R$ centered at $Y$. Moreover  $2^k\le r<2^{k+1}$.
The  set of all pairs
$(y,r)\in \cY\times \cR$ satisfying these conditions can be covered by
$\lc R^{d-1}2^{-k(d-1)}$ balls
(in $\bbR^{d+1}$) of radius $2^k$. Each such ball can contain only
$u 2^{k+1}$ pairs $(y,r)\in \cE_k$ by our assumption on $\cE_k$.
For each such $(y,r)$,  the scalar product
$\inn{F_{y,r}}{F_{Y,R}}$
is $O(2^{k(d-1)/2})$ by Lemma \ref{scalarlemma}. Consequently,
for fixed $(Y,R)$,
$$
\sum_{(y,r)\in\cE_k}\bigl|
\langle
F_{y,r},{F_{Y,R}}
\rangle\bigr|
\lc  R^{d-1}2^{-k(d-1)/2}\,u2^{k}\,,
$$
and therefore (as $N(u)=10\log_2(2+u)$)
$$\sum_{k:2^{N(u)}<2^k < R/4} \sum_{(y,r)\in\cE_k}\bigl|
\langle
F_{y,r},{F_{Y,R}}
\rangle\bigr|
\lc  R^{d-1} \sum_{k> N(u)} 2^{-k(d-1)/2}(u2^{k}) \lc R^{d-1};
$$
here we used that $d>3$ and  summed a decaying geometric progression
whose maximal term  corresponds
to $k= N(u)+10$.
Since $(d-1)/2>1$, we see that the geometric decay  cancels the large
factor  $u$ in the last displayed formula.
It remains to sum these estimates over pairs $(Y,R)$ to get the bound
$\sum_{(Y,R)\in \cE} R^{d-1}\lc \sum_k 2^{k(d-1)}\#\cE_k$ for  the sum of scalar products in
\eqref{sumofkscprod}.

Now that we have dealt with the interaction of incomparable radii,
we can concentrate on estimating $\|G_{k}\|_2^2$ for each $k$ separately.
It is convenient to
arrange the  radii  in intervals of length $u^a$ for some
$a>0$
and then apply the estimates of Lemma \ref{scalarlemma} to scalar products arising from different intervals; we shall see later that the choice of
$a=2/(d-1)$ is optimal.

Let
$I_{k,\mu}=[2^k+(\mu-1)u^{a},
2^k+\mu u^{a})$ for $\mu=1,2,\dots$,
 and let
$\cE_{k,\mu}$ be the set of all $(y,r)\in \cY\times I_{k,\mu}$ that
belong to  $\cE_k$.
Set
$$G_{k,\mu}=\sum_{(y,r)\in \cE_{k,\mu}} c(y,r) F_{y,r}.$$ We need to estimate the $L^2$ norm of $\sum_\mu G_{k,\mu}$.
By splitting the $\mu$ sum into ten different sums we may assume
that $\mu$ ranges over a $10$-separated set
and bound
$$
\Big\|\sum_\mu G_{k,\mu}\Big\|_2^2 \lc \sum_\mu \big\|G_{k,\mu}\big\|_2^2+
2 \sum_{\mu'> \mu} \big| \inn{G_{k,\mu'}}{G_{k,\mu}}\big|.
$$
Again, we shall first estimate the sum of the various scalar products, using
the assumption that the sets $\cE_k$ are of density type $(u,2^k)$.
We claim that
\begin{equation} \label{differentmu}
\sum_{\mu'> \mu} \big| \inn{G_{k,\mu'}}{G_{k,\mu}}\big|
\lc u^{1- a\frac{d-3}{2}} 2^{k(d-1)}\#\cE_k.
\end{equation}
To see this, we pick again some pair
$(Y,R) \in \cE_{k,\mu'}$ and examine  how
it interacts with pairs in $\cE_{k, \mu}$ where $\mu\le \mu'-10$.
 Note that if $(y,r)$ is such a pair
for which the scalar product is non-zero, then we must have $|y-Y|\le 2^{k+3}$
 and, since $|r-R|\le 2^{k+1}$,
we conclude that $|(y,r)-(Y,R)|\le 2^{k+4}$ in $\bbR^{d+1}$.
Moreover,
$|r-R| \ge u^{a}$ and thus the sum of the scalar products in which the
pair $(Y,R)$ participates is
$$
\lc 2^{k(d-1)}\sum_{\substack{(y,r)\in \cE_k:\\
u^{a}\le
 |(y,r)-(Y,R)|\le 2^{k+5}}}
|(y,r)-(Y,R)|^{-(d-1)/2}\,.
$$
Now we use the assumption that $\cE_k$ is of density type $(u,2^k)$
(\cf. \eqref{densitytype})
 and estimate the
displayed sum by
$$C_d 2^{k(d-1)}\sum_{2^\ell\ge u^a} (u2^\ell)  2^{-\ell\frac{d-1}2}
\lc
2^{k(d-1)} u^{1-a\frac{d-3}2};
$$
here  we have used  again that $d>3$.
We  sum over all $(Y,R)\in \cE_{k,\mu'}$ and then
 over all $\mu'$. Then the
left hand side of
\eqref{differentmu}
is $\lc u^{1-a \frac{d-3}{2} } 2^{k(d-1)}
\sum_\mu
\#\cE_{k,\mu};$  and \eqref{differentmu} follows.

We now estimate the  $L^2$ norm of each $G_{k,\mu}$.
For each $r\in \cR_{k,\mu}:=I_{k,\mu}\cap\cR$, let
$$G_{k,\mu,r}=\sum_{y:(y,r)\in \cE_k} c(y,r) F_{y,r}.$$
The conclusion of Lemma  \ref{scalarlemma} is now too weak to give satisfactory results; instead we apply the Cauchy-Schwarz inequality with respect to $r$
and use that the cardinality of
$\cR_{k,\mu}$ is $\lc u^{a}$.
Thus
$$\|G_{k,\mu}\|_2^2
\lc u^{a}
 \sum_{r\in \cR_{k,\mu}}
 \|G_{k,\mu,r}\|_2^2.  $$
Now
$G_{k,\mu,r}$ is the convolution of
$\sum_{y: (y,r)\in\cE_{k,\mu}}c(y,r)\psi_\o(\cdot-y)$ with
$\sigma_r*\psi_\o$.
By the standard decay estimate for the Fourier transform of the
surface
measure on the unit sphere, we have $$|\widehat \sigma_r(\xi)|\le
r^{d-1}
(1+r|\xi|)^{-\frac{d-1}2}
$$ and, since $\widehat \psi_\circ$ vanishes to high order at the
 origin,
we also  have, for $r\ge 1$,  \begin{equation}
\label{LinftyFourierbound}
\|\widehat{\sigma}_r\widehat\psi_\o\|_\infty \lc r^{(d-1)/2}.
\end{equation}
Since $\cY$ is
$1$-separated and the support of $\psi$ is contained in a ball of radius
$1/2$, we conclude that
$$
\|G_{k,\mu,r}\|_2^2\lc r^{d-1}\# \{y\in\cY: (y,r)\in \cE_{k,\mu}\}  $$
and thus
\begin{equation*}
\sum_{\mu}
\|G_{k,\mu}\|_2^2
\lc u^{a}
\sum_\mu \sum_{r\in \cR_{k,\mu}}
 \|G_{k,\mu,r}\|_2^2 \lc
u^{a}
2^{k(d-1)}\#\cE_k.
\end{equation*}
Combining this bound with
\eqref{differentmu} yields
$$
\big\|G_k\big\|_2^2 \lc \big(u^{a} +
u^{1-a\frac{d-3}{2}}\big) 2^{k(d-1)}\#\cE_k.
$$
The two terms balance if $a=2/(d-1)$  and with this
 choice  the previous bound becomes
$$
\big\|G_k\big\|_2^2 \lc
u^{\frac{2}{d-1} } 2^{k(d-1)}\#\cE_k.
$$
Finally,  we use this to estimate the first term in
\eqref{sumofkscprod}
and combine the resulting  bound
 with the earlier bound for the mixed terms in
\eqref{sumofkscprod}  to complete the proof of the lemma. \qed

\section{Application to compactly supported multipliers}
\label{compmult}
Now let  $m$ be a radial Fourier multiplier  supported in $\{1/2<|\xi|<2\}$
and let $K=\cF^{-1}[m]$. Since $K$ is radial, we can also write
$K= \kappa(|\cdot|)$ for some  $\kappa$. We shall prove the estimate
\begin{equation} \label{compsupp}\|K* f\|_p\lc \|K\|_p \|f\|_p,
\quad 1\le  p< p_d.
\end{equation}
Let  $\eta_\o$ be a radial Schwartz function whose Fourier transform
is supported in $\{1/4<|\xi|<4\}$ and such that $\widehat \eta_\o(\xi)=1$ on the
support  of $m$. Let $\psi_\o$ be a radial $C^\infty$
 function with compact support in $\{|x|\le 10^{-1}\}$ with the property that
 $\wh\psi_\o$ and all its derivatives up to order $20d$ vanish at the origin
but
$\wh\psi_\o (\xi)>0$
on $\{1/4\le |\xi|\le 4\}$ .
This is easy to achieve (take a radial function $\chi\in C^\infty_0$ such that
$\wh \chi(0) =1$, then define
$\psi_\o =\la^d \Delta^{10d}[ \chi(\la\cdot)]$ for a sufficiently 
large $\la$;  here $\Delta$ denotes the Laplacian in $\bbR^d$).

Let $\eta= \cF^{-1}[\wh
\eta_\o(\widehat \psi_\o)^{-2}]$.
Then
$K* f = \psi_\o*  K *  \psi_\o * g$
where $g=\eta*f$ and clearly $\|g\|_p\lc \|f\|_p$.
We split $K=K_0+K_\infty$ where $K_0=K \chi\ci{\{|x|\le 1\}}$.
Since
$\|K_0\|_1\lc \|K\|_p$
the operator of convolution
with $K_0$ is clearly bounded on all $L^p$, $1\le p\le \infty$,
with operator norm $O(\|K\|_p)$.
Therefore it suffices to show that
the $L^p$ norm of $\psi_\o*  K_\infty *  \psi_\o * g$ is controlled by
$C \|K\|_p \|g\|_p$.
We set $\psi=\psi_\circ*\psi_\circ$ and observe that
\begin{equation} \label{oneinftyintegral}
\psi*K_\infty*g= \int_1^\infty\int_{\bbR^d}\psi*\sigma_r (\cdot\!-y)\,
 \kappa(r)g(y)
dy \,dr\,.
\end{equation}
By Corollary \ref{mainlp},
$$\|\psi*K_\infty*g\|_p \lc\Big(\int|\ka(r)|^p r^{d-1}dr\Big)^{1/p}
\Big(\int|g(y)|^p dy\Big)^{1/p}.$$
This establishes \eqref{compsupp}.


\section{A variant of Corollary  \ref{mainlp} involving  large radii}
\label{largeradsect}

The following estimate for convolution operators with radial kernels will be used in conjunction with atomic decompositions to extend the one scale situation
of \S\ref{compmult} to the general case.
We consider radial kernels with cancellation that are supported in
$\{|x|>2^\ell\}$. The crucial feature is an  exponential gain in $\ell$,
 which will be useful when putting different scales together.
For  $\nu\in \bbZ$, let
$\cW^{\nu}$ be the  tiling of $\bbR^d$ with dyadic cubes of sidelength
 $2^{\nu}$, {\it i.e.},
the set of
 cubes of the form
$$[z_1 2^{\nu}, (z_1+1)2^{\nu}) \times \dots \times [z_d 2^{\nu}, (z_d+1)
2^{\nu}), \quad z=(z_1,\dots,z_d)\in \bbZ^d.$$

\begin{proposition}\label{germat} Let $1<p<p_d$ and
  $\eps<(d-1)(\frac 1p-\frac{1}{p_d})$. Let $\ell\ge 0$.
 Let $K$ be a radial convolution kernel supported in
$\{x:|x|>2^\ell\}$. For $s\in \bbZ$, let $K_s=2^{sd}K(2^s\cdot)$,
$\psi_s=2^{sd}\psi(2^s\cdot)$.
Then
\begin{equation}
\label{germatomic}
\big\|\psi_s* K_s* g\big\|_p
\lc \|K\|_p 2^{-\ell \eps}
\Big(\sum_{W\in \cW^{\ell-s}} \meas(W)\,
\|g\chi\ci W\|_\infty^p\Big)^{1/p}.
\end{equation}
\end{proposition}

The constant implicit in \eqref{germatomic} depends on  $\eps$.

We  prove  a variant of Corollary \ref{mainlp}, which involves
only
radii  $r\ge 2^\ell$ and corresponds to the case $s=0$ of the
proposition.
  Let $F_{y,r}$ be as in \eqref{gyr}.

\begin{lemma}\label{largerad}
Let $1<p<p_d$ and
  $\eps<(d-1)(\frac 1p-\frac{1}{p_d})$. Then, for $\ell\ge 0$,
\begin{multline}
\label{ellgain}
\Big\| \int_{\bbR^d}\int_{2^\ell}^\infty h(y,r) F_{y,r}\,dr\,dy\Big\|_p
\\\lc 2^{-\ell \eps}2^{\ell d/p}
\Big(\int_{r=2^\ell}^\infty \sum_{W\in \cW^\ell}
\sup_{y\in W}
|h(y,r)|^p r^{d-1}\,dr\Big)^{1/p}.
\end{multline}
\end{lemma}

\begin{proof}
We shall base the proof on the arguments in \S\ref{mainsect}
 and first prove a discretized version.
Let $\cY$, $\cR$ be  $1$-separated subsets of $\bbR^d$ and
 $[1,\infty)$ respectively.
Inequality \eqref{ellgain} follows from the following discretized version
by the averaging argument employed in the proof of Corollary \ref{mainlp}.
\begin{multline}
\label{ellgaindiscr}
\Big\| \sum_{\substack{(y,r)
\in \cY\times \cR\\ r\ge 2^\ell}}
\ga(y,r) F_{y,r}\Big\|_p
\\\lc 2^{-\ell \eps}2^{\ell d/p}
\Big(\sum_{r\in \cR} \sum_{W\in \cW^\ell}
\sup_{y\in \cY\cap W}
|\ga(y,r)|^p r^{d-1}\Big)^{1/p}.
\end{multline}
For $j\in \bbZ$, $r\in \cR$, let
 $\cW^\ell (j,r)$ be the set of all $W\in \cW^\ell$ for which
$2^j\le \sup_{x\in W}|\gamma(x,r)|< 2^{j+1}$. For each $y\in \cY$, let $W(y)$ be the unique cube in $\cW^\ell$ that
contains $y$, and for each $j\in \bbZ$, let $\cE_k(j)$ be the set of  all
$(y,r) \in \cY\times\cR_k$ with the property that
$W(y)\in \cW^\ell(j,r)$.
Apply  the density decomposition of Lemma \ref{udec}
to the sets $\cE_k(j)$ and write
$\cE_k(j)=\sum_{u\in \cU} \cE_k(j,u)$ as in that lemma.
Lemma \ref{lpbound}
applied to the set $\cup_{k\ge \ell} \,\cE_k(j,u)$
yields
\begin{equation}\label{applicationoflemma}
\Big\|\sum_{\substack{(y,r) \in \\\cup_{k\ge\ell}\cE_k(j,u)}} \gamma(y,r) F_{y,r} \Big\|_p^p
\lc  u^{- \delta p} 2^{jp} \sum_{k\ge \ell}\sum_{(y,r)\in \cE_k(j,u)}r^{d-1},
\end{equation}
for $\delta<\frac 1p-\frac{1}{p_d}$.
Now we use that
$\cE_k(j,u)$
is of density type $(u,2^k)$. Since  $k\ge \ell$, this implies that
for every $u\in \cU$,  every $j$, every $W\in \cW^\ell$, and every
$r\in [2^k,2^{k+1})$, the slice
$\cE_k(j,u,W,r)
:=\{y\in \cY\cap W: (y,r)\in
\cE_k(j,u)
\}$
contains 
$O(u 2^{\ell })$
 points. Also, since
 $\cY$  is  $1$-separated, the cardinality of each slice is $\lc 2^{\ell d}$.
Therefore the  right hand side of
 \eqref{applicationoflemma}
is controlled by
\begin{align*}&2^{jp} u^{-\delta p} \sum_{k\ge \ell} \sum_{r\in \cR_k} r^{d-1}
\sum_{W\in \cW^\ell}
\#\cE_k(j,u,W,r)
\\& \lc 2^{jp} C(\ell,u)
\sum_{k\ge \ell} \sum_{r\in \cR_k}r^{d-1} \#\cW^\ell (j,r),
\end{align*}
with
$
C(\ell,u):=
u^{-\delta p}
\min\{ u 2^{\ell}, \, 2^{\ell d}\}. $
By interpolation (Lemma \ref{dyadicinterpol}),
\begin{align*}&\Big\|\sum_j\sum_{(y,r) \in \cup_{k\ge\ell}\cE_k(j,u)} \gamma(y,r) F_{y,r} \Big\|_p^p
\\&\qquad\lc C(\ell,u)
\sum_j 2^{jp}
\sum_{k\ge \ell} \sum_{r\in \cR_k} r^{d-1} \#\cW^\ell (j,r)
\\&\qquad\lc  C(\ell,u)  \sum_{W\in \cW^\ell}
 \sum_{r\in \cR} r^{d-1}\sup_{y\in W} |\gamma(y,r)|^p.
\end{align*}
We sum  geometric progressions to  get
$\sum_{u\in \cU} C(\ell,u)^{1/p}
\lc 2^{-\ell\delta(d-1)} 2^{\ell d/p}$. Hence, with $\eps=(d-1)\delta$,
\begin{equation*}
\Big\|\sum_j\sum_{\substack{(y,r) \in \\ \cup_{k\ge\ell}\cE_k(j)}} \gamma(y,r) F_{y,r} \Big\|_p^p
\lc 2^{-\ell \eps p}  \sum_{r\in \cR} r^{d-1} \sum_{W\in
  \cW^\ell}\meas(W)
\sup_{y\in W} |\gamma(y,r)|^p.
\end{equation*}
This proves \eqref{ellgaindiscr}.
\end{proof}

\begin{proof}[Proof of Proposition \ref{germat}]
By scaling we may assume $s=0$.
 As in \S\ref{compmult}, we write
$$
\psi*K_\infty*g= \int_{2^\ell}^\infty\int_{\bbR^d}
\psi*\sigma_r (\cdot-y)\, \kappa(r)g(y)
\, dy \,dr.
$$
Apply  Lemma \ref{largerad} with $h(y,r)=\kappa(r)g(y) $
and notice that the right hand side of
\eqref{ellgain} is equal to
 $$2^{-\ell\eps}
\Big(\int_{2^\ell}^\infty| \kappa(r)|^p   r^{d-1} dr
\Big)^{1/p} \, \cdot\,
\Big(\sum_{W\in \cW^\ell} \meas(W) \|g\chi\ci W\|_\infty^p
\Big)^{1/p} \,.$$
\end{proof}

\section{Atomic decompositions and the proof of
Theorem \ref{mainthm}}\label{atomicsect}

The purpose of this chapter  is to prove Theorem \ref{mainthm} for one particular
Schwartz function $\eta$ whose Fourier transform is compactly
supported away from the origin (for the extension to  more general
$\eta$
see \S\ref{concl}).
We follow the presentation in \S\ref{main} and introduce a radial Schwartz function $\eta_\circ$ such that
$\wh\eta_\circ$ is supported in $\{\xi:1/2<|\xi|<2\}$ and satisfies
\begin{equation}\label{one} \sum_{s\in \bbZ} [\wh \eta_\circ(2^{-s}\xi)]^2=1
\end{equation} for all $\xi\neq 0$. Let $\psi_\circ $ be a $C^\infty$ function compactly supported in $\{x:|x|\le 1/10\}$ such that
$\wh\psi_\circ$ does not vanish in
$\{\xi:1/4\le|\xi|\le 4\}$ and  does vanish to
order $10 d$ at the origin. Let $\psi=\psi_\circ*\psi_\circ$ and
\begin{equation}\label{etadefinition}\eta=
 \cF^{-1}[\wh \eta_\circ/ \wh \psi].
\end{equation}
We shall use this particular  $\eta$ in the assumption of our theorem; in other words,   we shall assume that
$\sup_{t>0} \|T_m[ t^{d/p} \eta(t\cdot)]\|_p\le B_p<\infty.$
For $s\in \bbZ$, let
$$H_s= \cF^{-1} [\widehat \eta (\cdot) m (2^s\cdot)].$$
By our assumption,
\begin{equation}\sup_{s\in \bbZ} \|H_s\|_p \le B_p.
\label{Hsbd}
\end{equation}
Now let $K_s= 2^{sd}H_s(2^s\cdot)$, $\psi_s=2^{sd}\psi(2^s\cdot)=
2^{sd}(\psi_\circ*\psi_\circ)(2^s\cdot)$,
 $\eta_s= 2^{sd} \eta (2^s\cdot)$.
By \eqref{one} and our definitions,
we have the  decomposition
$$T_m f = \sum_s \psi_s*\psi_s * K_s * f_s$$
where
\begin{equation}\label{defoffs} f_s =\eta_s*f.
\end{equation}
We may assume that $f$ is a
 Schwartz function whose
Fourier transform is compactly supported away from the origin; this class
is dense in $L^p(\bbR^d)$, $1<p<\infty$. For those functions,  the sum
in $s$ is finite.

We shall work with atomic decompositions constructed from
Peetre's maximal square function
({\it cf.} \cite{peetre}, \cite{triebel}
and \cite{se-stud}) using ideas from work  by Chang and
Fefferman \cite{crf}.  The nontangential version of Peetre's expression is
$$
Sf (x) = \Big(\sum_s \sup_{|y|\le 10 d\cdot 2^{-s}}
|f_s(x+y)|^2\Big)^{1/2}.
$$
Then
the $L^p$ norm of $Sf$ is controlled  by $\|f\|_p$ if $1<p<\infty$,
 and by the Hardy space (quasi-)norm $\|f\|_{H^p}$ if $p\le 1$.
These statements follow, for example, from
the Fefferman-Stein
inequalities for the vector-valued
 Hardy-Littlewood maximal operator (\cite{fs}).

Put  $\Psi_s=\psi_s*\psi_s$. The proof of the $L^p$ boundedness of
$T_m$ reduces to the inequality
\begin{equation}\label{goal1}
\Big\|\sum_s\Psi_s*K_s* f_s\Big\|_p \lc B_p
\|Sf\|_p,\quad 1<p<p_d;
\end{equation}
here we now
assume that the sum in $s$ is over a finite set of integers.
 In what follows, we will make several decompositions of the
Schwartz functions $f_s$ (involving
even rough cutoffs)
and  the a priori convergence of various sums can be justified
by using the rapid decay of the functions.

The cancellation of the functions $\psi_s$ is crucial for the
estimation of the left hand side in \eqref{goal1} and
 various similar expressions. A simple tool is the inequality
\begin{equation}\label{almostorth}
\Big\|\sum_s \psi_s*h_s\Big\|_\tau \le C \Big(\sum_s
\|h_s\|_\tau^\tau\Big)^{1/\tau},
\quad 1\le \tau\le 2,
\end{equation}
with a constant $C$ depending only on $\psi$.
This is immediate from Plancherel's theorem for $\tau=2$,
trivial for $\tau=1$ and true by interpolation for $1<\tau<2$.
Inequality \eqref{almostorth} is not enough to put the estimates for the
various scales together, and in addition we have to use
an ``atomic decomposition''  of each $f_s$, which we now describe.

For fixed $s$, we tile $\bbR^d$ by the dyadic cubes of sidelength $2^{-s}$; and we shall  write $L(Q)=-s$
to indicate that the sidelength of a dyadic cube is $2^{-s}$.
For each integer $j$, we introduce the set
$\Omega_j=\{x: Sf(x)>2^j\}$. Let $\cQ^s_j$ be the set
of all dyadic cubes for which  $L(Q)=-s$  and which have the property that
$|Q\cap \Omega_j|\ge |Q|/2$ but
$|Q\cap \Omega_{j+1}|<|Q|/2$.
We also set
$$ \Omega_j^*= \{ x: M \chi\ci{\Omega_j}(x)>100^{-d}\}$$
where $M$ is the Hardy-Littlewood maximal operator. $\Omega_j^*$ is an open set
containing $\Omega_j$ and $|\Omega_j^*|\lc |\Omega_j|$.
We work with a Whitney decomposition $\cW_j$ of $\Omega^*_j$ into
dyadic cubes $W$. Specifically, $\cW_j$ is the set of all dyadic cubes $W$ such that the $20$-fold dilate of $W$ is contained in $\Omega_j^*$
and $W$ is maximal with respect to this property.
We note that each $Q\in \cQ^s_j$ is contained in a unique $W\in \cW_j$.
This is verified
 by showing that the  $20$-fold dilate $Q^*$ of $Q$
belongs to $\Omega_j^*$.
Indeed,
$|Q^*\cap \Omega_j|/|Q^*|\ge 20^{-d} |Q\cap\Omega_j|/|Q|\ge 40^{-d}$;
hence $Q^*\subset \Omega_j^*$.
We shall also need that the quadruple dilates $W^*$ of $W$, $W\in
\cW_j$,
have bounded overlap (uniformly in $j$).


We now define some building blocks that
are analogous to  the usual atoms; however they are not normalized,
and,  since we are mainly interested in
$L^p$ bounds
for $p>1$,  we do not insist on  cancellation. For each
$W\in \cW_j$,  set
$$A_{s,W,j} = \sum_{\substack{Q\in\cQ^s_j\\Q\subset W}} f_s \chi\ci Q;$$
note that only terms with $L(W)+s\ge 0$ occur.
We also need to consider ``cumulative atoms'', as any dyadic  cube $W$ can be a
Whitney cube for several $\Omega_j^*$. We set
$$A_{s,W}=\sum_{j: W\in \cW_j} A_{s,W,j}.$$
Note that
$$f_s= \sum_{W\in\cup_j\cW_j}
 A_{s,W}= \sum_j \sum_{W\in \cW_j} A_{s,W,j}.$$

The following observations about
atomic
decomposition are standard (see
\eg \cite{crf}), but included here for completeness.
\begin{lemma}
\label{atomsL2bd}
For each $j\in \bbZ$, the following inequalities hold.

(i)
$$
\sum_{W\in \cW_j} \sum_{s}\|A_{s,W,j}\|_2^2 \lc 2^{2j} \meas(\Omega_j).
$$

(ii) There is a constant $C_d$ such that for every  assignment
$W\mapsto s(W)$ defined on  $\cW_j$,
and every
$0\le p\le 2$,
$$
\sum_{W\in \cW_j} \meas(W)\|A_{s(W),W,j}\|_\infty^p \le C_d 2^{pj} \meas(\Omega_j).
$$
\end{lemma}
\begin{proof} Using
  the definitions of the atoms, part (i) follows from the inequality
$$
\sum_s\sum_{Q\in \cQ^s_j} \|f_s \chi\ci Q\|^2_2 \lc 2^{2j} \meas(\Omega_j).
$$
To see this,  observe that
$\meas(Q\setminus \Omega_{j+1}) \ge\meas(Q)/2$
for each $Q\in \cQ^s_j$,
and we also have
 $Q\subset \Omega_j^*$. We use this together with Fubini's theorem
and see that the left hand side of (i) is bounded by
\begin{align*}
&\sum_s\sum_{Q\in \cQ^s_j} \meas(Q)  \|f_s \chi\ci Q\|^2_\infty \le
\sum_s\sum_{Q\in \cQ^s_j} 2\,\meas(Q\setminus \Omega_{j+1})\,  \|f_s \chi\ci Q\|^2_\infty
\\
&\le2\int_{\Omega_j^*\setminus \Omega_{j+1}}
\Big[ \sum_s \sup_{|y|\le \sqrt d 2^{-s}}
|f_s(x+y)|^2\Big] dx \le 2\cdot 2^{2(j+1)} \meas(\Omega_j^*),
\end{align*}
which is  $\lc  2^{2j} \meas(\Omega_j)$.

Part (ii) of
the lemma follows since
$$
\|A_{s,W,j}\|_\infty \lc \sup_{\substack{Q\in \cQ^s_j\\Q\subset W}}
\big|f_s\chi\ci Q\big| \le \sup_{x\in \Omega_j^*\setminus \Omega_{j+1}} |Sf(x)| \le 2^{j+1}
$$
and
$\sum_{W\in \cW_j}|W|\le |\Omega_j^*|\lc |\Omega_j|.$
\end{proof}

To establish \eqref{goal1} we need to verify the inequality
\begin{equation}\label{all}
\Big\|\sum_{s,j} \sum_{\ell\ge 0}\sum_{\substack{W\in \cW_j\\
L(W)=-s+\ell}}
\Psi_s *K_s
* A_{s,W,j}\Big\|_p \lc B_p \|Sf\|_p.
\end{equation}

For each integer $\ell$ in this sum, we split
the convolution operator $K_s$ into  short range and  long range pieces,
$K_{s,\ell}^\sh$ and $K_{s,\ell}^\lo$. To define them,
we first look at  the
rescaled kernels $H_s$ and set
$ H_{s,\ell}^\sh(x)= H_s(x)$ if $|x|\le 2^{\ell}$ and
$ H_{s,\ell}^\sh(x)= 0$ if $|x|> 2^{\ell}$.
Also
$ H_{s,\ell}^\lo(x)= H_s(x)-H_{s,\ell}^\sh$.
Now set  $K_{s,\ell}^\sh= 2^{sd} H_{s,\ell}^\sh(2^s\cdot)$
and
$K_{s,\ell}^\lo= 2^{sd} H_{s,\ell}^\lo(2^s\cdot)$.
Finally, we split the sum in \eqref{all} into two parts, replacing $K_s$ by
$K_{s,\ell}^\sh$ and $K_{s,\ell}^\lo$, respectively.

Now consider  $W$ with $L(W)= -s+\ell$ and note that
 the short range convolution
$\psi_s *K_{s,\ell}^\sh
* A_{s,W,j}$ is supported in the quadruple  dilate $W^*$ of $W$;
thus for fixed $j$, all these terms are supported in $\Omega_j^*$. We
prove the short range inequality
\begin{equation}
\label{short}
\Big\|\sum_{s,j}
\sum_{\ell\ge 0} \sum_{\substack{W\in \cW_j\\L(W)=-s+\ell}}
\Psi_s *
K_{s,\ell}^\sh
*A_{s,W,j}
\Big\|_\tau \lc B_p \|Sf\|_\tau
\end{equation}
for $p<2d/(d+1)$, $\tau<2$.
The choice
$\tau=p$ is, of course, permitted for the $p$-range of Theorem \ref{mainthm}.
To prove \eqref{short}, it suffices  to show that for fixed $j$, and
for
$\tau\le 2$,
$p<2d/(d+1)$,
\begin{equation}\label{short-j}
\Big\|\sum_s
\sum_{\ell\ge 0}\sum_{\substack{W\in \cW_j\\
L(W)=-s+\ell}}
\Psi_s *
K_{s,\ell}^\sh* A_{s,W,j}
\Big\|_\tau^\tau\lc B_p^\tau
2^{j\tau}\meas(\Omega_j).
\end{equation}
Indeed, by Lemma \ref{dyadicinterpol}, inequality \eqref{short-j}
implies that the left hand side of \eqref{short} is controlled  for
$\tau<2$ by
$B_p^\tau \sum_j
2^{j\tau}\meas(\Omega_j) \lc  B_p^\tau\|Sf \|_\tau^\tau.$

Inequality \eqref{short-j} for $\tau<2$  follows from
 \eqref{short-j} for $\tau=2$  by H\"older's inequality. Here we use that
 the relevant expressions are supported in $\Omega_j^*$ and
 $|\Omega_j^*|\lc |\Omega_j|$.
To prove
\eqref{short-j} for $\tau=2$, we use a standard estimate
for  the Fourier transform of radial kernels
$K=\int_0^\infty \kappa(r)\sigma_r dr$, namely,
\begin{equation}\label{FTbound}
\|\wh K\wh\psi
\|_\infty \le C_p \|K\|_p = c\Big(\int_0^\infty |\kappa (r)|^p r^{d-1}
dr\Big)^{1/p},
 \qquad p<\frac{2d}{d+1}.
\end{equation}
Indeed using Bessel functions as in the proof of Lemma \ref{scalarlemma},
one can use H\"older's inequality to estimate
\begin{align*}
|\wh K(\xi)|&= c'\int_0^\infty \kappa(r) r^{d-1} B_d(r|\xi|) dr
\\
&\lc
\Big(\int_0^\infty |\kappa (r)|^p r^{d-1}
dr\Big)^{1/p} \Big(\int_0^\infty
r^{d-1}(1+r|\xi|)^{-\frac{d-1}{2}p'} dr\Big)^{1/p'}.
\end{align*}
It is easy to see that the last $L^{p'}$ norm is $O(|\xi|^{-d/p'})$,
 provided that $p<2d/(d+1)$. The  bound \eqref{FTbound}
follows since $\wh\psi$ is a Schwartz function that vanishes to high order at $0$.

We return to  \eqref{short-j} for $\tau=2$. As
$\Psi_s*K_{s,\ell}^\sh* A_{s,W,j}$ is supported in
$W^*$ and the $W^*$ have bounded overlap, we can
 dominate the left hand side of the inequality
by
\begin{equation}
\label{auxilL2est}
\begin{aligned}&\Big\|\sum_{W\in \cW_j}
\sum_s\psi_s*\psi_s*
K_{s,L(W)+s}^\sh* A_{s,W,j}
\Big\|_2^2
\\&\lc\sum_{W\in \cW_j}\Big\|\sum_s\psi_s*\psi_s*
K_{s,L(W)+s}^\sh* A_{s,W,j}
\Big\|_2^2
\\&\lc
\sum_{W\in \cW_j}\sum_s\big\|\psi_s*
K_{s,L(W)+s}^\sh* A_{s,W,j}
\big\|_2^2
\\
&\lc
 \sup_{s,\nu}\|\wh{\psi_s}
\widehat{ K_{s,\nu}^\sh}\|_\infty^2
\sum_{W\in\cW_j}
\sum_s \|A_{s,W,j}\|_2^2.
\end{aligned}
\end{equation}
Here we used the $L^2$ case of \eqref{almostorth}.
Now, by \eqref{FTbound}, the Fourier transform of
$\psi_s* K_{s,\nu}^{\sh}$ has $L^\infty$ norm
$\lc \|H_{s,\nu}^{\sh}\|_p\lc \|H_s\|_p\le B_p$. Thus, by Lemma
\ref{atomsL2bd}, (i),
 the last displayed quantity is $\lc B_p^2 2^{2j}|\Omega_j|$.
This finishes  the proof of \eqref{short-j}.

We now turn to the long range estimate, that is,
\begin{equation}\label{long}
\Big\|\sum_{s,j}
\sum_{\ell\ge 0}
\sum_{\substack{W\in \cW_j\\
L(W)=-s+\ell}}
\psi_s* \psi_s* K_{s,\ell}^\lo*
A_{s,W,j}\Big\|_p \lc  B_p \|Sf\|_p.
\end{equation}
We use the $j$-sum to combine the atoms into the cumulative atoms
$A_{s,W}$, take out the $\ell$-sum by Minkowski's inequality, and
use \eqref{almostorth}.
Thus the left hand side of \eqref{long}  is dominated by a
constant times
\begin{equation}\label{longterm}\sum_{\ell\ge 0} \Big(\sum_s \Big\|
\psi_s* K^\lo_{s,\ell} *
\sum_{W:
L(W)=-s+\ell} A_{s,W}\Big\|_p^p\Big)^{1/p}.
\end{equation}
 Now $\|H_{s,\ell}^\lo\|_p\le \|H_s\|_p\le B_p$
and therefore
Proposition
\ref{germat} implies that, for  fixed $\ell$,
\begin{multline}\label{longtermell}
\Big\|
\psi_s* K^\lo_{s,\ell} *
\sum_{W:
L(W)=-s+\ell}
A_{s,W}\Big\|_p\\
\lc
2^{-\ell\eps} B_p
\Big(\sum_{W: L(W)=-s+\ell} \meas(W) \,\|A_{s,W}\|_{\infty}^p\Big)^{1/p}
\end{multline}
for $p<p_d$, with some $\eps=\eps(p)>0$.
Note that for fixed $s,W$,
 the functions $ A_{s,W,j}$ live on disjoint sets (since the dyadic
 cubes of sidelength  $2^{-s}$ are disjoint and each is in exactly one
 family $\cQ_j^s$).
Thus, clearly,
$$\|A_{s,W}\|^p_\infty \lc\sum_j\|A_{s,W,j}\|^p_\infty.$$ It follows
  that the expression
\eqref{longterm} is
\begin{align*}
&\lc B_p\sum_\ell 2^{-\ell\eps}
\Big(\sum_j\sum_{W\in \cW_j}  \meas(W)
\|A_{\ell-L(W),W,j}\|_{L^\infty(W)}^p\Big)^{1/p}
\\
&\lc B_p\sum_\ell 2^{-\ell\eps}
\Big(\sum_j \meas(\Omega_j)\,2^{jp}\Big)^{1/p} \lc B_p
\|S f\|_p
\end{align*}
by part (ii) of
Lemma \ref{atomsL2bd}.
This yields  \eqref{long}. Finally, \eqref{all} follows from
\eqref{short} and \eqref{long}. This concludes the proof of the $L^p$
boundedness of $T_m$
 under the assumption
\eqref{Hsbd}.
\qed

\section{Conclusion of the proof} \label{concl}
We still have to prove \eqref{equiv} for an arbitrary choice of
$\eta$. To
this end, we
fix the radial multiplier $m$ and consider the family $\Theta$ of all
$C^\infty$
functions $\vphi$ compactly supported away from the origin such that
the condition
\begin{equation}\label{normalcond} \big\|\cF^{-1}[\vphi m(t\cdot)]\big\|_p<\infty
\end{equation} holds.
Note that if $\vphi\in\Theta$, then
$\vphi(\lambda\cdot)\in\Theta$ for every $\lambda>0$, moreover
 $\vphi\circ R \in\Theta$ for every rotation $R$ of $\bbR^d$ (here we use
the fact
that $m$ is radial). Also if $\chi$ is
any compactly supported $C^\infty$
function,  then  $\chi\vphi\in\Theta$, simply because
 $\chi$ is an $\cF L^p$ multiplier. Finally  if
$\vphi_1,\vphi_2\in\Theta$, 
then
 $\vphi_1+\vphi_2\in\Theta$.

Now assume that there exists at least one not identically zero function
$\vphi_\circ\in\Theta$.
Let $V$ be a non-empty open subset of $\bbR^{d+1}$ such that
$|\vphi_\circ|>0$ on $V$. Let $\vphi$
be any other $C^\infty$ function compactly supported away from the origin.
For every $\xi\in\bbR^d\setminus\{0\}$, one can find a rotation $R_\xi$ and a number $\lambda_\xi>0$
such that $\lambda_\xi R_\xi \xi\in V$ or, equivalently, $\xi\in \lambda_\xi^{-1}R_\xi^{-1}V$. Then the open sets
$
\lambda_\xi^{-1}R_\xi^{-1}V
$,
$\xi\in\supp\,\vphi$, form a  cover of $\supp\,\vphi$.
Choose a finite subcover
$
\lambda_{\xi_j}^{-1}R_{\xi_j}^{-1}V
$,
$j=1,\dots, n$, and put
$$
\zeta=\sum_{j=1}^n\overline{\vphi_\circ(\lambda_{\xi_j}R_{\xi_j}\cdot)}
\vphi_\circ(\lambda_{\xi_j}R_{\xi_j}\cdot)\,.
$$
Note that $\zeta\in\Theta$ and $\zeta>0$ on
$
\bigcup_{j=1}^n\lambda_{\xi_j}^{-1}R_{\xi_j}^{-1}V\supset\supp\,\vphi\,.
$
Hence, the function $\chi$ defined as $\vphi/\zeta$ on
$\supp\,\vphi$
and $0$ on
$\bbR^d\setminus\supp\,\vphi$
is a $C^\infty$ function with compact support, so
$\vphi=\chi\zeta\in\Theta$.

\begin{proof}[Proof of Theorem \ref{mainthm}, concluded]
Let $g$ be an arbitrary Schwartz function, then the
condition
$\sup_{t>0} \|T_m[ t^{d/p} g(t\cdot)]\|_p<\infty$ is clearly necessary for
$L^p$ boundedness. Conversely, suppose that this condition is satisfied;
it is equivalent to
$\sup_{t>0} \|\cF^{-1}[m(t\cdot) \widehat g]\|_p<\infty$.
We may pick $\chi\in C^\infty$ with compact support in $\bbR^d\setminus \{0\}$ so that
$\chi \widehat g $ is not identically $0$. Since $\chi$ is a
Fourier multiplier, we  see that  $\chi  \widehat  g \in \Theta$. By the
above considerations we also have $\wh \eta \in \Theta$
where $\eta $ is as
in \eqref{etadefinition}. But for this $\eta$, the characterization is
 already proved and the $L^p$ boundedness of $T_m$ follows. \end{proof}

\section{Variants and extensions}\label{extvar}

\noi{\bf Hardy space estimates.}
We now give an extension of Theorem \ref{mainthm} to
the range $p\le 1$. We prove,
in  dimensions $d\ge 2$,  a full characterization  of the
convolution operators with radial kernels mapping the Hardy space
$H^p$ to $L^p$.

\begin{theorem}
\label{hardythm}
Suppose  $d\ge 2$ and $0<p\le 1$.
Let  $m$ be  radial and let
$\eta$ be a Schwartz function whose Fourier transform is compactly
supported  away from the origin and is not identically $0$.
Then
$$
\big\|T_m\big\|\ci{H^p\to L^p}\,\asymp \,\sup_{t>0}\,t^{d/p}\big
\|T_m[\eta(t\cdot)]\big\|\ci{L^p}\,.
$$
\end{theorem}

\noi {\it Remarks. (i)}
The $H^p\to L^p$ boundedness
is equivalent to $H^p\to H^p$ boundedness,
by Theorem 3.4 in \cite{miy2}.

\noi{\it (ii)}   The proof is substantially simpler than the $L^p$ result
for $p>1$; in particular,  the crucial
 orthogonality Lemma \ref{scalarlemma} plays no
role, and is replaced by the $L^\infty$ multiplier  bound
\eqref{LinftyFourierbound}.
This allows to include
dimensions two and three.



\begin{proof}[Sketch of proof of Theorem \ref{hardythm}]
We first note  the chain of inequalities
\begin{equation*}
\|\widehat H\|_\infty \le \|H\|_1\le
\sum_{z\in \bbZ^d} \sup_{y\in [0,1]^d} |H(z+y)|
\le
\Big(\sum_{z\in \bbZ^d} \sup_{y\in [0,1]^d} |H(z+y)|^p\Big)^{1/p}
\end{equation*}
since $p\le 1$. Now note that if $\widehat H$ is supported in
 $\{|\xi|\le 2\}$,
 then the last expression is $O(\|K\|_p)$,
by a Plancherel-P\'olya type estimate
({\it cf.} \cite{triebel}, \S 1.3.3).

 Now the proof of the short range estimate \eqref{short} for $\tau\le 1$
is rather similar to the argument in \S\ref{atomicsect}.
Note that $\Psi_s *
K_{s,\ell}^\sh
*A_{s,W,j}$ is supported in $W^*\subset \Omega_j^*$. Thus we can bound the left
hand side of \eqref{short}, for $\tau \le 1$,  by
\begin{align*}
\Big (\sum_j |\Omega_j|^{1-\tau/ 2}
\Big\| \sum_{\substack{W\in \cW_j}}\sum_s
\sum_\ell \Psi_s *
K_{s,\ell}^\sh *A_{s,W,j}
\Big\|_2^{\tau}\Big)^{1/\tau}.
\end{align*}
By \eqref{auxilL2est}
and Lemma  \ref{atomsL2bd}(i), this is dominated by
$$\sup_{s,\nu}\|H^{\text sh}_{s,\nu}\|_p
\Big (\sum_j |\Omega_j|^{1-\tau/ 2} (2^{2j}|\Omega_j|)^{\tau/2} \Big)^{1/\tau}
$$
which is $\lc \sup_s\|H_s\|_p \|Sf\|_\tau$. Of course, we may choose 
$\tau=p$.

We prove  the analogue of the  long range estimate \eqref{long}.
As $p\le 1$, we can apply the triangle inequality for the $p$-th power
of the $L^p$-(quasi)-norm for the sums in $s$, $\ell$, $j$ and $W$.
After rescaling to the case $s=0$,
matters are reduced to the estimation of the convolution with a radial
kernel $\int_{r\ge 2^\ell} \kappa(r) \sigma_r*\psi_0\, dr$  where
$\kappa(|\cdot|)$ is the Fourier transform of a
function supported in $\{1/2<|\xi|< 2\}$. The relevant estimate is
then
\begin{multline}\label{keyestimatecont}
\Big\|\int_{2^\ell}^\infty\kappa(r) \sigma_r* \psi_0*A_{0,W,j}\, dr\Big\|_p
\\
\lc 2^{-\ell \eps(p)}
\Big(\int_{2^\ell}^\infty|\kappa (r)|^p r^{d-1} dr\Big)^{1/p}
 |W|^{1/p}
 \|A_{0,W,j}\|_\infty
\end{multline}
where 
$|W|=2^{\ell d}$.
Now let $\ka_n^*= \sup_{n\le r\le n+1} |\kappa(r)|.$
We shall establish
\begin{multline}\label{keyestimate}
\Big\|\int_{r>2^\ell}\kappa(r) \sigma_r* \psi_0*A_{0,W,j} \Big\|_p\\
\lc 2^{-\ell \eps(p)}
(\sum_{n\ge 1} |\ka^*_n|^p n^{d-1})^{1/p}
 |W|^{1/p}
 \|A_{0,W,j}\|_\infty
\end{multline}
and \eqref{keyestimatecont} will  follow  by the
Plancherel-P\'olya type estimate
$$(\sum_{n\ge 1} |\ka^*_n|^p n^{d-1})^{1/p} \lc_p
\Big(\int|\kappa (r)|^p r^{d-1} dr\Big)^{1/p}.
$$

We now prove \eqref{keyestimate},  with $\eps(p)= (d-1)(\frac1p-\frac 12)$.
Since $p\le 1$, the left hand side is dominated by
$$\Big(
\sum_{n\ge 2^\ell}
\Big\|\int_{n}^{n+1} \kappa(r)
\sigma_r* \psi_0*A_{0,W,j}\,dr \Big\|_p^p\Big)^{1/p}.
$$
As $n\ge 2^\ell$, the term $ \sigma_r* \psi_0*A_{s,W,j}$, for $n\le
r\le n+1$,  is supported
in an annulus with width $c2^\ell$ and inner and outer radii
comparable to $n$, hence of measure $\lc n^{d-1}2^\ell$. By
\eqref{LinftyFourierbound},
$$\sup_\xi\Big|\int_{n}^{n+1}\kappa(r)
\widehat\sigma_r(\xi)\widehat \psi_0(\xi)
 \,dr \Big|\lc
|\ka^*_n| n^{(d-1)/2}.$$
We use  H\"older's inequality and estimate
$\| \int_{n}^{n+1}\kappa(r)\sigma_r* \psi_0*A_{0,W,j}\, dr \|_p $ by
\begin{align*}
&(n^{d-1} 2^\ell)^{\frac 1p-\frac12}
\Big\| \int_{n}^{n+1} \kappa(r)\sigma_r* \psi_0*A_{0,W,j} \,dr\Big\|_2
\\&\lc
\ka^*_n n^{\frac{d-1}{p}} 2^{\ell(\frac1p-\frac12)}
\|A_{0,W,j} \|_2.
\end{align*}
But
$\|A_{0,W,j} \|_2
\lc 2^{\ell d/2} \|A_{0,W,j} \|_\infty
$,
and therefore the last displayed expression is controlled by
\begin{align*}&\ka^*_n n^{\frac{d-1}{p}}
2^{\ell(\frac{d-1}{2}+\frac1p)}  \|A_{0,W,j} \|_\infty
\\
&\lc \ka^*_n n^{\frac{d-1}{p}}
2^{-\ell(d-1)(\frac 1p-\frac12)}|W|^{\frac 1p}
  \|A_{0,W,j} \|_\infty.
\end{align*}
Finally, we remark that the arguments in \S\ref{concl} carry over to
the $H^p$ case, $p\le 1$.
\end{proof}

\medskip

\noi{\bf Lorentz space estimates.}
Weak type $(p,p)$ ({\it i.e.}, $L^p\to L^{p,\infty}$) estimates
 for convolutions with radial kernels, in
particular for Bochner-Riesz means,
have been considered in  \cite{taowt} and the references therein.
We shall indicate here how to prove $L^p\to L^{p,\nu}$ estimates
by combining our previous arguments with interpolation by the real method
(the general Marcinkiewicz theorem).
We will use the  following simple fact  about Lorentz spaces.

\begin{lemma}\label{fubinilemma}
Let $(\cX_1,\mu_1)$, $(\cX_2,\mu_2)$ be $\sigma$-finite
measure spaces,
and let
$\mu=\mu_1\times\mu_2$ be
 the product measure on $\cX_1\times\cX_2$. Then, for $1\le
p<\infty$,  $p\le \nu\le \infty$, and any $\mu$-measurable function  $G$,
\begin{equation}\label{fubinilor}
\|G\|_{L^{p,\nu}(\cX_1\times\cX_2,\mu)}\le C_{p,\nu}
\Big(\int \|G(x_1,\cdot)\|_{L^{p,\nu}(\cX_2,\mu_2)}^p d\mu_1\Big)^{1/p}.
\end{equation}
\end{lemma}
\begin{proof}
Let $p\le \nu<\infty$. By Fubini's theorem, the Lorentz space norm
$\|G\|_{L^{p,\nu}(\cX_1\times\cX_2)}$ is controlled by
\begin{align*}
\Big(\int_0^\infty \alpha^{\nu -1}
\Big[\int_{\cX_1} \mu_{2}(\{x_2\in\cX_2:|G(x_1,x_2)|>\alpha\})\,d\mu_1\Big]^{\nu/p} d\alpha
\Big)^{1/\nu}.
\end{align*}
By Minkowski's inequality, this is bounded by
$$\Big(\int_{\cX_1} \Big(\int_0^\infty \alpha^{\nu -1}
[ \mu_{2}(\{x_2\in\cX_2:|G(x_1,x_2)|>\alpha\})]^{\nu/p} d\alpha
\Big)^{p/\nu} d\mu_1\Big)^{1/p},
$$
which is comparable
to the right hand side of \eqref{fubinilor}. The case $\nu=\infty$ is similar.
\end{proof}

We state   a result only
for multipliers that are compactly
supported
away from the origin.

\begin{theorem}  \label{multcomplor}
Let $d\ge 4$, $1<p<p_d=\frac{2d-2}{d+1}$, $p\le \nu\le \infty$,
 and let  $m$ be  radial and supported in $\{\xi:1/2\le |\xi|\le 2\}$.
Then
\begin{align}\label{weaktypenu}
\big\|T_m\big\|\ci{L^p\to L^{p,\nu}}
&\,\asymp \,\|\widehat
m\|\ci{L^{p,\nu}}\,,
\\
\label{strongtypenu}
\big\|T_m\big\|\ci{L^{p,\nu}\to L^{p,\nu}}
&\,\asymp \,\|\widehat m\|\ci{L^{p}}\,.
\end{align}
\end{theorem}

\begin{proof}
The lower bound for the operator norm in \eqref{weaktypenu}
 follows in the
usual way, by testing on suitable  Schwartz functions. By Colzani's
theorem (\cite{colzani}) for convolution operators,  the
$L^p\to L^p$
operator norm is
controlled  by the $L^{p,\nu}\to L^{p,\nu}$
operator norm and this implies the lower
bound for  \eqref{strongtypenu}.

For the upper bounds, we apply real interpolation to  the second
 inequality in Corollary \ref{mainlp}
and obtain
\begin{equation}\label{mainlpnucont}
\Big\| \int_{\bbR^d}\int_0^\infty
h(y,r) F_{y,r} \, dr dy \Big\|_{L^{p,\nu}(\bbR^d)} \lc
\|h\|_{L^{p,\nu}(\bbR^d\times [1,\infty);dy\, r^{d-1}dr)}.
\end{equation}

Now let $K=\widehat m$. We argue as in \S\ref{compmult}.
Split $K=K_0+K_\infty$.
Then $\|K_0\|_1\lc \|K_0\|_{L^{p,\nu}}$  and therefore
$\|K_0*f\|_{L^{p,\nu}}\lc \|K\|_{L^{p,\nu}} \|f\|_{L^{p,\nu}}$.
To estimate the main term $K_\infty*f =\psi* K_\infty *g$ we express
it as in
\eqref{oneinftyintegral} and then apply \eqref{mainlpnucont}.
Using  Lemma \ref{fubinilemma} we can estimate
$\|\psi* K_\infty *g\|_{L^{p,\nu}}$ by
either
$$\|\ka\|_{L^{p,\nu}(\bbR^+, r^{d-1}dr)} \|g\|_{L^p(\bbR^d)} =
C\|K\|_{L^{p,\nu}(\bbR^d)} \|g\|_{L^p(\bbR^d)},
$$
or by
$$\|\ka\|_{L^{p}(\bbR^+, r^{d-1}dr)} \|g\|_{L^{p,\nu}(\bbR^d)} =
C\|K\|_{L^{p}(\bbR^d)} \|g\|_{L^{p,\nu}(\bbR^d)}.
$$
\end{proof}

\noi{\it Remark:}  One can also obtain  $L^p\to L^{p,\nu}$ estimates for 
multipliers that  are not necessarily compactly supported.
However the proper generalization of the  $L^{p,\nu}\to L^{p,\nu}$ bound in 
\eqref{strongtypenu} presents some difficulties at the current stage. 
We hope  to consider these and related matters later.

\section{The  regularity result for the wave equation}
\label{lsmsect}

%
In this section we shall prove
Theorem \ref{lsmsobolev}. We first  note
 that by a standard scaling argument, it suffices to prove
the inequality
\begin{equation}\label{normalizedSob}
\Big(\int_1^2\big\|e^{it\sqrt{-\Delta}} f\big\|^q_q \, dt\Big)^{1/q}
\lc \|(I-\Delta)^{\alpha/2} f\|_q.
\end{equation}
Indeed, let us
first show how
\eqref{unscaledSob}
follows assuming \eqref{normalizedSob} (here $q<\infty$). We may assume by symmetry that in
\eqref{unscaledSob} we integrate over $[0,L]$. We then write
\begin{align*} &\Big(L^{-1}\int_0^L\|e^{it\sqrt{-\Delta}}f\|_q^q
\, dt\Big)^{1/q}
\,\le \sum_{n=1}^\infty
\Big(L^{-1}\int_{2^{-n}L}^{2^{-n+1}L}\|e^{it\sqrt{-\Delta}}f\|_q^q
\, dt\Big)^{1/q}
\\
&= \sum_{n=1}^\infty 2^{-n/q}
\Big(\int_{1}^{2}\|e^{iL2^{-n}s\sqrt{-\Delta}}f\|_q^q
\, ds\Big)^{1/q}=
\sum_{n=1}^\infty 2^{-n/q} \big(*)_n
\end{align*}
where
\begin{equation*}
(*)_n=
\Big(\int_{1}^{2}\int_{\bbR^d} \big|[e^{is\sqrt{-\Delta}}f_{L,n}] (L^{-1}2^n x)\big|^q dx \,ds\Big)^{1/q} \text{ and } f_{L,n}(y)= f(L2^{-n}y).
\end{equation*}
We change variables in $x$, apply  \eqref{normalizedSob},
and then change variables again to see that
$$(*)_n \lc
(L2^{-n})^{d/q} \|(I-\Delta)^{\alpha/2}f_{L,n}\|_q=
\|(I-2^{-2n}L^2 \Delta)^{\alpha/2}f\|_q
.$$
Now
we have for  $\alpha\ge 0$, $n\ge 0$,
$$\|(I- 2^{-2n}L^2\Delta)^{\alpha/2}f\|_q
\le C_q \|(I- L^2\Delta)^{\alpha/2}f\|_q$$
where $C$ does not depend on $L$ and $n$; for $1<q<\infty$, this  follows, for
example, from the Mikhlin-H\"ormander multiplier theorem.
Thus
$(*)_n$ is bounded by the right hand side of
\eqref{unscaledSob} uniformly in $n\ge 1$, and, for $q<\infty$,  
the sum $\sum_{n=1}^\infty 2^{-n/q}(*)_n$ is essentially
 dominated by the same quantity.

We shall actually obtain  an improvement of \eqref{normalizedSob},
which is formulated using dyadic decompositions.
Let $\eta_\circ $ be  as in \eqref{one}. 
Define           
$P_k$  by $\widehat{P_k f}=
(\widehat \eta_\circ(2^{-k}\xi))^2 \widehat f$ for $k>0$ and
$P_0=I-\sum_{k\ge 1}P_k$.
We
have chosen $k$ as our index for the dyadic frequency pieces
instead of $s$, firstly to distinguish it from the homogeneous
expression ($s\in \bbZ$) used earlier and, secondly,  to match it
with the notation in \S\ref{mainsect};  the term for
large frequencies $\approx 2^k$ will correspond, after an appropriate
 rescaling, to
the situation of Corollary \ref{mainlp} when the radii are taken
in $[2^k, 2^{k+1}]$.

\begin{theorem} \label{triebellizthm}
Suppose $d\ge 4$, $\frac{2d-2}{d-3}<q<\infty$, and
$\alpha=d(\frac 12-\frac 1q)-\frac 12.$
Then
\begin{equation}\label{vectorineq}
\Big(\int_1^2\Big\|\sum_{k\ge 0} |P_k e^{it\sqrt{-\Delta}}f|
\Big \|_{q}^q \,dt\Big)^{1/q}
\lc \Big(\sum_{k\ge 0}2^{k\alpha q}\big\|P_k f\big\|_{q}^q
\Big)^{1/q}.
\end{equation}
\end{theorem}
The slightly weaker inequality for Sobolev spaces follows if we
replace the $\ell^1$
norm in $k$ on the left hand side of \eqref{vectorineq}
and the $\ell^q$ norm
on the right hand side (with $q>2$)  by the $\ell^2$ norms.
Inequality  \eqref{vectorineq} can be restated using Triebel-Lizorkin
spaces, namely,
\begin{equation*}
\Big(\int_1^2\big\|e^{it\sqrt{-\Delta}}f\big\|_{F^q_{0,1}}^q \,dt\Big)^{1/q}
\lc \|f\|\ci{F^{q}_{\alpha,q}}.
\end{equation*}


It will
be convenient to dispose of the terms corresponding to  $k=0,1$.
 Let
$\chi\ci 0$ be a radial $C^\infty_0(\bbR^d)$ function such  that
$\chi\ci 0(\xi)=1$ for $|\xi|\le 1$ and $\chi\ci 0(\xi)=0$ for $|\xi|\ge
3/2$. One easily checks that $\chi\ci 0(\xi/\la) e^{i|\xi|}$ is the
Fourier transform of an $L^1$ function for any $\la$ (with $L^1$
norm growing in $\la$ for $\la\to \infty$). Indeed, the contribution of
the multiplier near the origin is handled by considering
$m_\ka(\xi)=(\chi\ci 0(2^\ka \xi)-\chi\ci 0(2^{\ka+1}\xi))
(e^{i|\xi|}-1)$. One bounds the derivatives of
$m_\ka(2^{-\ka}\xi)$ for $\ka>0$ to see that the $L^1$ norm of
$\F^{-1}[m_\ka]$  is $ O(2^{-\ka})$.

Next,
we  describe a  further  reduction to  an
inequality involving spherical means (\cf.   \eqref{mainlsm}, \eqref{mukt} below).
This can be done in various ways. One way is to apply the method of
stationary phase in conjunction with
multiplier theorems.
We will give a more direct approach based on the principle that every
radial function can be written as an average of spherical
measures. As before, we let
$\sigma_\rho $ denote  the surface  measure on the sphere of radius $\rho$.

Let $\vth$ be a $C^\infty$-function on the real line supported in
$(1/8,8)$ such that $\vth(s)=1$ on $(1/4,4)$.
  For $k\ge 1$, define  the convolution kernel $K_k$  by
$$\widehat {K_k}(\xi)
= e^{i|\xi|}  \vth(2^{-k}|\xi|).$$

\begin{lemma} \label{averagelemma}
Let $d\ge 2$. Then, for $k\ge 1$,
\begin{equation}\label{reprofKk}K_k= 2^{k(d-1)/2} \int_{1/2}^2
  w_k(\rho) \sigma_\rho d\rho\,+\, E_k\end{equation}
where
\begin{equation} \label{uniformL1}
\sup_k \int_{1/2}^2|w_k(\rho)|d\rho<\infty,
\end{equation}
and, for any $M$,
\begin{equation}\label{Ekdecay}\|E_k\|_1\le
C_{M,d} 2^{-kM }.
\end{equation}
\end{lemma}
\begin{proof}
 We use polar coordinates for the Fourier integral defining $K_k$ and then write an integral over the sphere
 $S^{d-1}$
in terms of integrals over $d-2$ dimensional spheres perpendicular to
 $x$.
We get
\begin{align*} (2\pi)^d  K_k(x)&=
\int_{\bbR^d}\vth(2^{-k}|\xi|) e^{i|\xi|}e^{i\inn{\xi}{x} }d\xi
\\&= 2^{k(d-1)} \int_0^\infty
 \vth(2^{-k}s) (2^{-k} s)^{d-1} e^{is} \int_{S^{d-1}}
e^{is|x|\inn{\frac{x}{|x|}}{\theta} } d\sigma(\theta) \, ds
\\&=c_{d-2}
2^{k(d-1)} \int_{-1}^1
 2^k\Theta(2^k(1+\tau|x|)) (1-\tau^2)^{\frac{d-3}{2}} d\tau,
\end{align*}
where $c_{d-2}$ is the surface measure of the unit sphere $S^{d-2}$ and
$$\Theta (\sigma)= \int_0^\infty\vth(s) s^{d-1}
e^{is\sigma} ds.$$
Clearly $\Theta\in \cS(\bbR)$.

From the above formula it is clear that
\eqref{reprofKk} holds with $$w_k(\rho)=c_{d-2}(2\pi)^{-d}
2^{k\frac{d-1}2} \int_{-1}^1
 2^k\Theta(2^k(1+\tau\rho)) (1-\tau^2)^{\frac{d-3}{2}} d\tau$$
and $E_k(x)= 2^{k\frac{d-1}2} w_k(|x|) [1- \chi\ci{[1/2,2]}(|x|)]$.

Let $\gamma>-1$ be fixed  and let $\Theta$ be any Schwartz function on $\bbR$ whose Fourier transform is supported in $(1/8,8)$. We prove that for $\beta\ge 1$, $\rho>0$,
\begin{equation}\label{decaybound}
\int_{-1}^1
\Theta(\beta(1+\tau\rho )) (1-\tau^2)^{\gamma} d\tau \le C
\beta^{-\gamma-1}  (1+\beta|1-\rho|)^{-N}
\end{equation} for any $N         >1$.
Here
$C\ge 0$ depends on $\Theta,\gamma, N$ but not on $\beta$ or $\rho$.
Clearly,  the  $L_1((0,\infty))$ norm of the right hand side 
of \eqref{decaybound} is $O(\beta^{-\gamma-2})$. Thus 
 \eqref{decaybound} applied with  $\gamma=\frac{d-3}2$
and $\beta=2^k$  yields  the bounds
\eqref{uniformL1} and \eqref{Ekdecay}.


The bound   \eqref{decaybound} is straighforward; one  examines
 separately the three  cases
$0<\rho<\frac 12$,
$1/2\le \rho\le 1$, and $\rho>1$. We may assume that $N\ge 1$.

Let $C_0=\sup_{x\in\mathbb R}|\Theta(x)|(1+2|x|)^{N+\gamma+2}$.
Then, for  $0<\rho<\frac 12$, the integral can be estimated by
$$C_0\int_{-1}^1(1-\tau^2)^\gamma d\tau \,(1+\beta)^{-N-\gamma-2}$$
which is better than the claimed bound.

If $1/2\le \rho\le 1$, we split the integral over $[-1,1]$ as
$\int_{-1}^0+
\int_0^1$. For the latter, we may  argue as in the previous case  and
bound it
by the last displayed expression.
For the integral over  $[-1,0]$, we make the change of
 variable $\tau=-1+t$, set
$C_1=\sup_{x\in\mathbb R}|\Theta(x)|(1+|x|)^{N+\gamma+2}$ and  bound
$|\int_{-1}^0\cdots d\tau|$ by
$$C_1 \int_0^1 \frac{t^{\gamma}(2-t)^\gamma}
{(1+\beta(1-\rho)+\beta\rho t)^{N+\gamma+2} }\,dt
\le C_2 (\beta\rho)^{-\gamma-1} (1+\beta(1-\rho))^{-N-1}
$$
where
$C_2= C_1\max\{1, 2^\gamma\}
\int_0^\infty t^\gamma(1+t)^{-N-\gamma-2} dt$.

Finally we consider the last case, $\rho\ge 1$.
Here we use the fact that the Fourier transform of $\Theta$ is
supported
in $( 1/8,8)$ and thus $\Theta$ extends to an entire function satisfying
 $|\Theta(x+iy)| \le C e^{-y/8} $ for $y\ge 0$.
Now set $g_\gamma(z)= (1-z^2)^\gamma$ so that $g_\gamma$ is analytic
in the upper half plane
and $g_\gamma(x)$ is nonnegative  for $x\in [-1,1]$.
By Cauchy's theorem and  limiting arguments,
the integral over the real line  of $\Theta g_\gamma$ vanishes and
therefore
$$\Big|\int_{-1}^1\Theta(x) g_\gamma(x) dx\Big|=\Big|\int_{\mathbb R\setminus
  [-1,1]}\Theta(x) g_\gamma(x) dx\Big|\,.$$
The latter integral is bounded by
$$
C_3\int_1^\infty \frac{(\tau^2-1)^\gamma}
{[1+\beta(\tau\rho-1)]^{N+2\gamma+2}}\,d\tau
=
C_3\int_0^\infty \frac{[t(2+t)]^\gamma}{[1+\beta(\rho-1)+\beta\rho
    t]^{N+2\gamma+3}}\,dt
$$
where
$C_3=2\sup_{x\in\mathbb R}|\Theta(x)|(1+|x|)^{N+2\gamma+3}$.
One separately considers the cases $\gamma\ge0$ and $-1<\gamma<0$. It
is not hard to see that in both cases the last displayed expression
can be estimated by
$$C_3\max\{1,4^\gamma\} \int_0^\infty t^\gamma (1+\beta(\rho-1)+\beta\rho
    t)^{-N-\gamma-2}dt
$$ which in turn is equal to
$$C_4 (\beta\rho)^{-\gamma-1} (1+\beta(\rho-1))^{-N-2}$$
with $C_4=C_3\max\{1,4^\gamma\}
\int_0^\infty\frac{t^\gamma}{(1+t)^{N+\gamma+2}}dt$.
\end{proof}

We  continue with the proof of \eqref{vectorineq}.
Let  $K_{k,t}= t^{-d} K_k (t^{-1}\cdot)$ with $K_k$ as in the
lemma and observe that
$$P_k [e^{it\sqrt{-\Delta}} f]= P_k[K_{k,t}* f], \quad 1/2\le t\le 2.$$
We first dispose of the error terms $E_k$. Let $E_{k,t}=t^{-d}E_k(t^{-1}\cdot)$. 
Then for any fixed $t\in [1,2]$
$$\Big\|\sum_{k\ge 0}|E_{k,t}*P_k f|\Big\|_q\lc \sum_{k\ge 0} 
2^{-kM} \|P_kf\|_q$$
which, by H\"older's inequality, is controlled by the right hand side of 
\eqref{vectorineq}.

Now define
\begin{equation}\label{mukt} \mu_{k,t} =
\int_{1/2}^2 w_k(\rho)\sigma_{\rho t} d\rho,
\end{equation}
with $w_k$ satisfying \eqref{uniformL1}.
In view of Lemma \ref{averagelemma}, it suffices to prove that,
 for $q>q_d$,
 the estimate
\begin{equation}
\label{fkpestimate}\Big(\int_{1}^2\Big\|\sum_{k=2}^\infty 2^{k\frac{d-1}2}\big| \mu_{k,t}
 *\psi_k*f_k\big|
\Big\|_q^q dt\Big)^{1/q}
\lc \Big(\sum_{k} \|f_k\|_q^q 2^{kq
(d(\frac12-\frac 1q)-\frac12)}
 \Big)^{1/q};
\end{equation}
holds for all $\{f_k\}_{k=2}^\infty $ with  $\widehat f_k$  supported
in $\cA_k:=\{\xi: 2^{k-1}<|\xi|<2^{k+1}\}$. Here $\psi_k$ are suitably chosen so that
  $\psi_k = 2^{kd}\psi (2^k\cdot)$, $\psi=\psi_\circ*\psi_\circ$,
$\psi_\circ$ is supported in
$\{|x|\le 10^{-1}\}$ with $10 d$ vanishing moments
(see the discussion leading to  \eqref{etadefinition}).
In addition we assume that
 $\wh\psi_\o(\xi)\neq 0$ for $1/2\le|\xi|\le 4$.
To see how \eqref{fkpestimate} implies \eqref{vectorineq} we choose
$f_k=2^{k(d-1)/2}L_k f$ with $\widehat{L_k f}(\xi)
=\eta_\circ^2(2^{-k}\xi)
[\wh\psi(2^{-k}\xi)]^{-1} \wh f(\xi)$ and use that $\zeta/\widehat \psi$
is the Fourier transform of a Schwartz function
for every  $\zeta$ that 
is smooth and compactly supported in $\{\xi: 1/3\le
|\xi|\le 3\}$.

It  suffices to prove \eqref{fkpestimate}
for families $\{f_k\}$
for which all but finitely many of the $f_k$ are zero,
with constant  independent of the number of summands.
By duality the desired bound then  follows from
\begin{multline}\label{mainlsm}
\Big(\sum_{k=2}^\infty  2^{k \frac{d}{p'}p}\Big\|\int_1^2 \mu_{k,t} *\psi_k
*g_k(\cdot, t) \,dt \Big\|_p^p
\Big)^{1/p}\\
\lc \Big(\int_1^2\big\|\sup_k |g_k(\cdot,t)|\big\|_p^p \,
dt\Big)^{1/p},\quad p<p_d,
\end{multline}
for all $\{g_k\}_{k=2}^\infty$, with the property
that  the (spatial) Fourier transform of $g_k(\cdot,t)$
is supported in $\cA_k$.

To prove \eqref{mainlsm}, we need the following inequality for fixed $k$
(which will be a straightforward consequence of
Lemma \ref{largerad}).
Let $\cW^{\ell-k}$ denote the set of dyadic cubes of sidelength
$2^{\ell-k}$.
\begin{proposition}\label{lsmdualW}
  Let $1\le p<p_d$ and $\eps<(d-1)(\frac1p-\frac{1}{p_d})$. Then, for  $0\le
  \ell\le k$,
\begin{multline}\label{muktestimate}
\Big\|\int_1^2 \psi_k*\mu_{k,t}* g(\cdot, t)\, dt\Big\|_p\\
\lc_\eps
2^{-k d/p'} 2^{-\ell\eps} \Big(\sum_{W\in \cW^{\ell-k}}|W|
\int_1^2 \sup_{y\in W}|g(y,t)|^p dt\Big)^{1/p}.
\end{multline}
\end{proposition}

\begin{proof}
We first prove the inequality
\begin{multline}\label{rhoeqone}
\Big\|\int_1^2 \psi_k*\sigma_{t}* g(\cdot, t)\, dt\Big\|_p\\
\lc_\eps
2^{-k d/p'} 2^{-\ell\eps} \Big(\sum_{W\in \cW^{\ell-k}}|W|
\int_1^2 \sup_{y\in W}|g(y,t)|^p dt\Big)^{1/p}.
\end{multline}
We  apply a  rescaling and averaging argument to deduce it from
Lemma \ref{largerad}.
Define $H_{k,t}$ by $\widehat {H_{k,t}}(\xi) =\widehat \psi(\xi)
\widehat \sigma_1 (2^k t\xi)$. The expression  on the
left hand side of \eqref{rhoeqone}
 can be written as
\begin{align*}
&\Big\|\int_1^2 2^{kd} H_{k,t}(2^k\cdot)* g(\cdot,t)\,t^{d-1}dt\Big\|_p
=2^{-kd/p}
\Big\|\int_1^2 H_{k,t}* g(2^{-k}\cdot,t)\, t^{d-1}dt\Big\|_p
\\
&=2^{-kd/p}
\Big\|\int_{2^k}^{2^{k+1}}
\psi*\sigma_r * 2^{-kd}g(2^{-k}\cdot,
2^{-k }r)
 dr\Big\|_p.
\end{align*}
By Lemma \ref{largerad},
 the last expression is
$$\lc 2^{-kd/p} 2^{-\ell\eps}2^{\ell d/p}\times
\Big(\int_{2^k}^{2^{k+1}}\sum_{W'\in\cW^\ell}|W'|
\sup_{y'\in W'} |2^{-kd}g(2^{-k}y', 2^{-k}r)|^p r^{d-1} dr
\Big)^{1/p},
$$
which is dominated by a constant times
\begin{align*}
& 2^{-\ell\eps} \Big(\int_1^2 \sum_{W\in \cW^{\ell-k}} 2^{-kd(p-1)}|W|
\sup_{y\in W} |g(y, t)|^p dt\Big)^{1/p}
\\
&\lc 2^{-\ell\eps}2^{-kd/p'} \Big(\sum_{W\in \cW^{\ell-k}} |W|
\int_1^2
\sup_{y\in W} |g(y,t)|^p dt\Big)^{1/p}.
\end{align*}

It remains to show how \eqref{rhoeqone} implies the assertion of the proposition. Since $\int|w_k(\rho)|d\rho$ is uniformly bounded it suffices, by averaging, to show the uniform bound
\begin{multline}\label{fixedrho}
\Big\|\int_1^2 \psi_k*\sigma_{\rho t}* g(\cdot, t)\, dt\Big\|_p\\
\lc_\eps
2^{-k d/p'} 2^{-\ell\eps} \Big(\sum_{W\in \cW^{\ell-k}}|W|
\int_1^2 \sup_{y\in W}|g(y,t)|^p dt\Big)^{1/p}, \qquad \frac 12\le
\rho\le 2.
\end{multline}
This is a consequence of \eqref{rhoeqone}
by scaling. For the details, assume $\rho\in (1,2]$.
After a change of variables we have to estimate the $L^p$ norm of
$$\Big(\int_{\rho}^2+\int_2^{2\rho}\Big) \big[\psi_k*\sigma_1* g(\cdot, \rho^{-1} t)
 \big](x)
\frac {dt}\rho.
$$

We apply
\eqref{rhoeqone} with the function
$g(\cdot, \rho^{-1} t)\chi_{[\rho,1]}(t)$ to bound the first integral.
The second integral is equal to
$$\frac{2}{\rho}\int_1^\rho \big[\psi_k*\sigma_{2s}*
g(\cdot, \tfrac{2s}{\rho})\big] (x)  ds= \frac{2^d}{\rho} \int_1^\rho
\big[\psi_{k+1}*\sigma_s* g (2\cdot, \tfrac{2s}{\rho})\big] \big(\frac x2) ds
$$
and, after conjugation with a dilation operator, we may apply
\eqref {rhoeqone} (with $\psi_k$ replaced by  $\psi_{k+1}$).
Note that  replacing
$\cW^{\ell-k}$ with
$\cW^{\ell-k-1}$ on the right hand side of
\eqref{fixedrho} yields an equivalent norm.
The argument for $\rho\in [1/2,1)$ is similar.
\end{proof}

We  now use the arguments of \S \ref{atomicsect}
based on  ``atomic''  decompositions for the functions $g_k(\cdot,t)$,
{\it for any fixed}  $t\in [1,2]$.
We work with  the $\ell^\infty$ variant of  Peetre's
operator,
namely,
$$
\cM G (x,t) = \sup_{k>0} \sup_{|y|\le 10 d\cdot 2^{-k}}
|g_k(x+y,t)|,
$$
where it will always be understood that  $G=\{g_k\}_{k=1}^\infty$ and
$g_k(\cdot,t)$ has spectrum in the annulus $\cA_k$.
Then,  with this specification, Peetre's inequality says that
 \begin{equation}\label{peetreinfty}\big\|\cM G(\cdot,
   t)\big\|_{L^p(\bbR^d)}
\lc
\|\sup_k |g_k(\cdot,t)| \,\|_{L^p(\bbR^d)},  \quad 0<p\le \infty.\end{equation}

 For each $t\in [1,2]$, let
$$\Omega_j(t)=\{x\in \bbR^d: \cM G(x,t)>2^j\}.$$ 
 Let $\cQ_j(t)$
be the set
of all dyadic cubes
 which are contained in $\Omega_j(t)$ but not in $\Omega_{j+1}(t)$.

For each dyadic cube of sidelength less than $1$ 
we define an expanded cube  $W(Q,t)$ as follows. We first
let $j(Q)$ be the unique $j$ such that $Q\in \cQ_j(t)$. If the unique
dyadic cube of sidelength $1$ containing $Q$  is contained in 
$\Omega_{j(Q)}(t)$, then we let
$W(Q,t)$ be this cube. If not then we let $W(Q,t)$ be the maximal dyadic
cube that contains $Q$ and that is contained in $\Omega_{j(Q)}(t)$.

We let $\cQ^k_j(t)$ be the family of cubes in $\cQ_j(t)$ which are of sidelength $2^{-k}$. Notice that if $Q$ has sidelength $2^{-k}$  then  the sidelength of $W(Q,t)$ is $2^{\ell-k}$ for some
nonnegative integer $\ell\le k$. As before we denote by
$\cW^{\ell-k}$ the collection of dyadic cubes of sidelength
$2^{\ell-k}$.
We also let $\cW_j(t)$ be the set of dyadic cubes contained in
$\Omega_j(t)$
which are either of sidelength $1$, or of sidelength less than $1$ and maximal
in $\Omega_j(t)$. Notice that the cubes in $\cW_j(t)$ have disjoint
interiors.
With these notations we note that if $Q\in \cQ^k_j(t)$ and $W(Q,t)$
has sidelength $2^{\ell-k}$  then $W(Q,t)$ is a cube in $\cW_j(t) \cap
\cW^{\ell-k}$.

For each $\ell=0,\dots, k$, define
$$A_{k,\ell,j}(x,t) = \sum_{\substack{Q\in\cQ^k_j(t)\\W(Q,t)\in
    \cW^{\ell-k}}}
g_k(x,t)
\chi\ci Q(x)\,.$$

We can now  decompose $$g_k=\sum_{\ell\ge 0}
 \sum_{j\in \bbZ} A_{k,\ell,j}.$$ Using this
decomposition and Minkowski's inequality,
we estimate the left hand side of \eqref{mainlsm} by
\begin{align*}&\sum_{\ell\ge 0}
\Big(\sum_k 2^{k dp/p'}
\Big\|\int_1^2 \mu_{k,t} *\psi_k *
\sum_j
A_{k,\ell,j}(\cdot,t) \,dt
\Big\|_p^p\Big)^{1/p}
\end{align*}
and, by Proposition \ref{lsmdualW}, the term corresponding
to a fixed $\ell$  is
\begin{align}
 &\lc 2^{-\ell\eps} \Big(\sum_k \sum_{W\in \cW^{\ell-k}}\meas(W)
\int_1^2 \sup_{y\in W}\Big|\sum_{j}
A_{k,\ell,j}(y,t)\Big|^p dt\Big)^{1/p}\notag
\\
&\lc 2^{-\ell\eps} \Big(\sum_k \sum_{W\in \cW^{\ell-k}}\meas(W)
\int_1^2 \sup_{y\in W}\sum_{j}\Big|
A_{k,\ell,j}(y,t)\Big|^p dt\Big)^{1/p}.
\label{finalsum}
\end{align}
where for the last estimate we have used that for each fixed $k$, $\ell$, $t$ the functions
$y\mapsto A_{k,\ell,j}(y,t)$, $j\in \bbZ$,  live on (essentially)
disjoint sets. 

To estimate \eqref{finalsum} we set, for  $W\in \cW_j(t)$,
$$A^W_{k,j}(\cdot,t)= \sum_{\substack{Q\in\cQ^k_j(t)\\W(Q,t)=W}}
g_k(x,t)
\chi\ci Q(x)\,
$$
so that
$A_{k,\ell,j}= \sum_{W\in \cW^{\ell-k}} A^W_{k,j}$.
By the definitions of $\cM$ and $\Omega_j$ we have
$\|A^W_{k,j}(\cdot,t)\|_\infty^p\le 2^{(j+1)p}$  for any $W\in \cW_j(t)$.
Therefore we get, for any fixed
$\ell$,
\begin{align*}
&\sum_k
\sum_{W\in  \cW^{\ell-k}}
\meas(W)\cdot\|A_{k,\ell,j}(\cdot,t)\chi\ci W \|_\infty^p 
\notag
\\&=
\sum_k
\sum_{W\in  \cW^{\ell-k}\cap \cW_j(t)}
\meas(W)\cdot\|A_{k,j}^W(\cdot,t)\|_\infty^p 
\notag
\\
&\le 2^{(j+1)p}\sum_{W\in \cW_j(t)} \meas(W)\,\le  2^{(j+1)p} \meas(\Omega_j(t)).
\end{align*}
The expression \eqref{finalsum} is now $\lc 2^{-\ell\eps}(*)_\ell$ 
where
\begin{align*}
(*)_\ell&:= \Big(\sum_j\int_1^2 \sum_k \sum_{W\in \cW^{\ell-k}} \meas(W)
\big\|A_{k,l,j}(\cdot,t)\chi\ci W\big\|_\infty^p\,dt\Big)^{1/p}
\\
& \lc  \Big(\int_1^2 \sum_{j} 2^{jp} \meas(\Omega_j(t))
dt\Big)^{1/p}
\lc  \Big(\int_1^2 \big\| \cM G(\cdot, t)\big\|_p^p
dt\Big)^{1/p}.
\end{align*}
We sum in $\ell$ and use \eqref{peetreinfty}
 to conclude the proof
of   \eqref{mainlsm}.
\qed

\end{document}